\documentclass[12pt]{amsart}
\usepackage{amsthm, amssymb}
\usepackage{amsfonts, color, amscd,mathabx}
\usepackage{epsfig,multicol}
\usepackage{xypic}
\theoremstyle{plain} \numberwithin{equation}{section}
\newtheorem{thm}{Theorem}[section]
\newtheorem{cor}[thm]{Corollary}
\newtheorem{prop}[thm]{Proposition}
\newtheorem{lem}[thm]{Lemma}

\DeclareMathOperator{\IM}{Im}
\theoremstyle{definition}
\newtheorem{defn}{Definition}
\newtheorem{exam}[thm]{Example}
\theoremstyle{remark}
 \newtheorem{rem}{Remark}

\def\Z{\Bbb Z}

\def\R{\Bbb R}

\begin{document}
\title[Self-dual binary codes from small covers and simple polytopes]{\large \bf Self-dual binary codes from small covers and simple polytopes}
\author[Bo Chen, Zhi L\"u and Li Yu]{Bo Chen, Zhi L\"u and Li Yu}
\keywords{Self-dual code, polytope, small cover.}
 \subjclass[2010]{57S25, 94B05, 57M60, 57R91}
\thanks{Supported in part by grants from NSFC (No.\,11371093, No.\,11371188, No.\,11401233, 
  No.\,11431009 and No.\,11661131004) and the PAPD (priority academic program development) of Jiangsu higher education institutions.}

\address{School of Mathematics and Statistics, Huazhong University of Science and Technology, Wuhan, 430074, P. R.  China}
\email{bobchen@hust.edu.cn}
\address{School of Mathematical Sciences, Fudan University, Shanghai,
200433, P.R. China.} \email{zlu@fudan.edu.cn}
\address{Department of Mathematics and IMS, Nanjing University, Nanjing, 210093, P.R.China}
\email{yuli@nju.edu.cn}


\begin{abstract}
We explore the connection between simple polytopes and self-dual binary codes
 via the theory of small covers. We first show that a small cover
$M^n$ over a simple $n$-polytope $P^n$ produces a self-dual code in the sense of Kreck--Puppe if and only if $P^n$ is $n$-colorable and $n$ is odd.
Then we show how to describe such a self-dual binary code in terms of the combinatorial information of $P^n$.
 Moreover, we can define a family of
 binary codes $\mathfrak{B}_k(P^n)$, $0\leq k\leq n$, from an arbitrary simple
$n$-polytope $P^n$. We will give some necessary and sufficient conditions for
 $\mathfrak{B}_k(P^n)$ to be a self-dual code. A spinoff of our study of such binary 
 codes gives some
  new ways to judge whether a simple $n$-polytope
  $P^n$ is $n$-colorable in terms of
  the associated binary codes $\mathfrak{B}_k(P^n)$.
  In addition, we prove that the minimum distance of
  the self-dual binary code obtained from a $3$-colorable simple $3$-polytope
  is always $4$.

\end{abstract}

\maketitle


\section{Introduction}\label{int}

 A (linear) \emph{binary code} $C$ of length $l$ is a linear subspace of the $l$-dimensional linear space
${\Bbb F}_2^l$ over ${\Bbb F}_2$ (the binary field).
 The \emph{Hamming weight} of an element
  $u=(u_1, \ldots, u_l)\in {\Bbb F}_2^l$, denoted by $wt(u)$,
   is the number of nonzero coordinates $u_i$ in $u$.
   Any element of $C$ is called a \emph{codeword}.
   The \emph{Hamming distance} $d(u,v)$ of any two codewords $u,v\in C$ is defined by:
   $$d(u,v)=wt(u-v).$$
  The minimum of the Hamming distances $d(u, v)$ for all $u, v\in C$ , $u\neq v$, is
called the \emph{minimum distance} of $C$ (which also equals
 the minimum Hamming weight of nonzero elements in $C$). A
 binary code $C\subset {\Bbb F}_2^l$ is called \emph{type} $[l,k,d]$ if $\dim_{{\Bbb F}_2} C =k$
 and the minimum distance of $C$ is $d$. We call two binary codes in ${\Bbb F}_2^l$
  \emph{equivalent} if they differ only by a permutation of coordinates.\vskip .1cm

 The standard bilinear form
    $\langle\ ,\, \rangle$ on ${\Bbb F}_2^l$ is defined by
$$\langle u, v\rangle :=\sum_{i=1}^l u_iv_i,\
 u=(u_1, \ldots, u_l), v=(v_1, \ldots, v_l)\in {\Bbb F}_2^l.$$
 Note that $\langle u , v\rangle = \frac{1}{2} \big(wt(u) + wt(v)-wt(u+v)\big)
  \ \text{mod} \ 2$ for any $u,v\in {\Bbb F}_2^l$, and
  $$\langle u, u\rangle=\sum_{i=1}^l u_i,\ u=(u_1, \ldots, u_l)\in {\Bbb F}_2^l.$$
 Then any linear binary code $C$ in ${\Bbb F}_2^l$ has a \emph{dual code} $C^{\perp}$ defined by
  $$C^\perp :=\{u\in {\Bbb F}_2^l\,|\, \langle u, c\rangle=0 \text{ for all $c\in C$}\}$$
  It is clear that $\dim_{{\Bbb F}_2} C + \dim_{{\Bbb F}_2}C^{\perp} =l$.
 We call $C$ {\em self-dual} if $C=C^\perp$.
 For a self-dual binary code $C$, we can easily show the following
 \begin{itemize}
   \item The length $l=2\dim_{{\Bbb F}_2}C$ must be even; \vskip .1cm
   \item For any $u\in C$,
  the Hamming weight $wt(u)$ is an even integer since $\langle u, u \rangle=0$;\vskip .1cm
  \item The minimum distance of $C$ is an even integer.
  \end{itemize}

 Self-dual binary codes play an important role in coding theory and have been studied
 extensively (see~\cite{rs} for a detailed survey).
 \vskip .1cm

Puppe in~\cite{p} found an interesting connection between
 closed manifolds and self-dual binary codes.
  It was shown in~\cite{p}
  that an involution $\tau$ on an odd dimensional closed manifold $M$ with ``maximal number
   of isolated fixed points'' (i.e., with only isolated fixed
points and the number of fixed points
   $|M^{\tau}|=\dim_{{\Bbb F}_2}(\bigoplus_i H^i(M;{\Bbb F}_2))$) determines a
  self-dual binary code of length $|M^{\tau}|$. Such an involution $\tau$ is called an
    \emph{$\mathrm{m}$-involution}. Conversely,
    Kreck--Puppe~\cite{kp} proved a somewhat surprising theorem that any self-dual binary
    code can be obtained from an $\mathrm{m}$-involution on some closed
    $3$-manifold.
   Hence it is an interesting problem for us to search
     $\mathrm{m}$-involutions on closed manifolds. But in practice
     it is very difficult to construct all possible $\mathrm{m}$-involutions
    on a given manifold. \vskip .1cm

    On the other hand, Davis and Januszkiewicz in~\cite{dj} introduced a class of closed smooth
  manifolds $M^n$ with locally standard actions of elementary 2-group $\Z_2^n$, called
   \emph{small covers}, whose orbit space is an
    $n$-dimensional simple convex polytope $P^n$ in $\R^n$. It was shown in~\cite{dj} that
    many geometric and topological properties of $M^n$ can be explicitly
    described in terms of the combinatorics of $P^n$ and some characteristic function
 on $P^n$ determined by the $\Z_2^n$-action.
  For example, the mod 2 Betti numbers of $M^n$ correspond to the $h$-vector of $P^n$.
   Any nonzero element $g\in\Z_2^n$ determines a nontrivial involution on $M^n$, denoted by
   $\tau_g$. We call $\tau_g$ a \emph{regular involution} on the small cover. So
   whenever $\tau_g$ is an $\mathrm{m}$-involution on $M^n$ where $n$ is odd,
    we obtain a self-dual binary code from $(M^n,\tau_g)$. \vskip .1cm

  Motivated by Kreck--Puppe and Davis--Januszkiewicz's work, our purpose in this paper is
  to explore the connection between the theory of binary codes and
  the combinatorics of simple polytopes via the topology of small covers.
  We will show that a small cover $M^n$ over an $n$-dimensional simple polytope $P^n$ admits
  a regular $\mathrm{m}$-involution only when
  $P^n$ is $n$-colorable.
  A polytope is called \emph{$n$-colorable} if we can color all the facets (codimension-one faces) of the polytope
   by $n$ different colors so that any
  neighboring facets are assigned different colors.
  Moreover, we find that the self-dual binary code
  obtained from a regular $\mathrm{m}$-involution on $M^n$
  depends only on the combinatorial structure  of $P^n$ and the parity of $n$.
 This motivates us to define a family of binary codes $\mathfrak{B}_k(P^n)$,
 $0\leq k \leq n$, for any simple polytope $P^n$ (not necessarily $n$-colorable).

\vskip .1cm

   The paper is organized as follows. In section 2, we 
   explain the procedure of obtaining self-dual binary codes as described in~\cite{p}
   from $\mathrm{m}$-involutions on closed manifolds. 
   In section 3, we first recall some basic facts of small covers and then
  investigate what kind of small covers can admit regular $\mathrm{m}$-involutions
  (see Theorem~\ref{max-involution}). In section 4, we
  spell out the self-dual binary code from a small cover with a regular
  $\mathrm{m}$-involution (see Corollary~\ref{Cor:Main-2}).
  It turns out that the self-dual
  binary code depends only on the combinatorial structure of the
  underlying simple polytope. In section 5, we
 study the properties of a family of binary codes $\mathfrak{B}_k(P^n)$,
 $0\leq k \leq n$, associated to any simple $n$-polytope $P^n$.
  A spinoff of our study produces some
  new criteria to judge whether $P^n$ is $n$-colorable in terms of
  the associated binary codes $\mathfrak{B}_k(P^n)$ (see Proposition~\ref{collection}).
  In section 6, we will give some necessary and sufficient conditions for
 $\mathfrak{B}_k(P^n)$ to be self-dual codes for general simple polytops $P^n$ 
 (see Theorem~\ref{thm:Self-Dual}). 
  In section 7, we prove that the minimum distance of the
   self-dual binary code obtained from any $3$-colorable simple $3$-polytope is always
    $4$ (see Proposition~\ref{prop:3-polytope}).
  In section 8, we investigate some special properties of
   $n$-colorable simple $n$-polytopes. In section 9, we study what kind of
  doubly-even binary codes can be obtained from $n$-colorable simple $n$-polytopes.
  In particular, we show
  that the extended Golay code and
  the extended quadratic residue code cannot be obtained from
   any $n$-colorable simple $n$-poltyopes.\vskip .6cm

  \section{Binary codes from $\mathbf{m}$-involutions on manifolds}\label{general-case}
 Let $\tau$ be an involution on a closed
 connected $n$-dimensional manifold $M$, which has only isolated fixed points.
Let $G_{\tau} \cong \Z_2$ denote the binary group generated by $\tau$.
By Conner~\cite[p.82]{cf}, the number $|M^{G_{\tau}}|$ of the
 fixed points of $G_{\tau}$ must be even. So we assume that $|M^{G_\tau}|=2r$, $r\geq 1$, in the following discussions. \vskip .1cm
By~\cite[Proposition(1.3.14)]{ap}, the following statements are equivalent.

\begin{itemize}
\item[(a)]  $|M^{G_{\tau}}|=\sum_{i=0}^n b_i(M; \Bbb F_2)$ (i.e. $\tau$ is
an $\mathrm{m}$-involution); \vskip .1cm
\item[(b)]  $H^*_{G_\tau}(M; {\Bbb F}_2)$ is a free $H^*(BG_\tau; {\Bbb F}_2)$-module, so
  $$H^*_{G_\tau}(M; {\Bbb F}_2)=H^*(M;{\Bbb F}_2)\otimes H^*(BG_\tau; {\Bbb F}_2);$$
\item[(c)] The inclusion of the fixed point set, $\iota: M^{G_{\tau}}\hookrightarrow M$, induces a monomorphism
$$\iota^*: H^*_{G_\tau}(M; {\Bbb F}_2)\rightarrow H^*_{G_\tau}(M^{G_\tau}; {\Bbb F}_2) \cong
{\Bbb F}^{2r}_2 \otimes {\Bbb F}_2[t].$$
\end{itemize}

 Next we assume that $\tau$ is an $\mathrm{m}$-involution on $M$.
 So
 the image of $H^*_{G_{\tau}}(M; {\Bbb F}_2)$ in ${\Bbb F}^{2r}_2 \otimes {\Bbb F}_2[t]$ under the
 localization map $\iota^*$ is isomorphic to $H^*_{G_{\tau}}(M; {\Bbb F}_2)$ as graded algebras.
 It is shown in~\cite{cl,p} that the image $\iota^*(H^*_{G_{\tau}}(M; {\Bbb F}_2))$
 can be described
 in the following way. For any vectors $x=(x_1,\ldots,x_{2r})$ and $y=(y_1,\ldots,y_{2r})$ in
${\Bbb F}_2^{2r}$, define
   $$x\circ y=(x_1y_1, \ldots , x_{2r}y_{2r}).$$
It is clear that ${\Bbb F}_2^{2r}$ forms a commutative ring with respect to
two operations $+$ and $\circ$. Actually, $({\Bbb F}_2^{2r}, +, \circ)$ is a boolean ring.
Notice that $x\circ x =x$ for any $x\in {\Bbb F}_2^{2r}$. Let
\begin{equation} \label{Equ:Big-V}
   \mathcal{V}_{2r}=\big\{x=(x_1,\ldots,x_{2r})\in
{\Bbb F}_2^{2r} \,\big| \, \langle x, x\rangle=\sum\limits_{i=1}^{2r}x_i=0 \in {\Bbb F}_2\big\}.
\end{equation}

  Then  $\mathcal{V}_{2r}$ is a $(2r-1)$-dimensional linear
  subspace of ${\Bbb F}_2^{2r}$. Note that for any $u\in \mathcal{V}_{2r}$, the Hamming weight
  $wt(u)$ of $u$ is an even integer. The following lemma is immediate from our definitions.

\begin{lem} \label{code}
 Let $C$ be a binary code in ${\Bbb F}_2^{2r}$ with $\dim_{{\Bbb F}_2} C =r$.
 Then the following statements are equivalent.
 \begin{enumerate}
\item[(C1)]
 $C$ is self-dual;
 \item[(C2)]  $\langle x, y\rangle=0$ for any $x, y\in C$;
 \item[(C3)] $x\circ y\in \mathcal{V}_{2r}$ for any $x, y\in C$.
 \end{enumerate}
\end{lem}

   Moreover, let
 \begin{equation} \label{Equ:V_i}
     V^M_k = \{ y\in {\Bbb F}_2^{2r}  \,\big|\, y\otimes t^k \in \mathrm{Im}(\iota^*)\}
     \subset {\Bbb F}_2^{2r}, \  k=0,\ldots, n.
 \end{equation}
 By the localization theorem for equivariant cohomology (see~\cite{ap}), we have isomorphisms
 \begin{equation} \label{Equ:Isom-Cohom}
   H^k(M^n;{\Bbb F}_2) \cong V_k^M\slash V_{k-1}^M, \ 0\leq k \leq n.
 \end{equation}
 \vskip .1cm

\begin{thm} [{\cite[Theorem 3.1]{cl}} or {~\cite[p.213]{p}}] \label{ring}
 For any $0\leq k \leq n$, we have
$$\dim_{{\Bbb F}_2} V^M_k=\sum\limits_{j=0}^k b_j(M;\Bbb F_2).$$
 In addition, $H^*_{G_\tau}(M^n;{\Bbb F}_2)$ is isomorphic to
the graded ring
$$\mathcal{R}_M=V^M_0+ V^M_1t+\cdots+
V^M_{n-2}t^{n-2}+ V^M_{n-1}t^{n-1}+ {\Bbb
F}_2^{2r}(t^{n}+t^{n+1}+\cdots)$$ where the ring structure of
$\mathcal{R}_M$ is given by
\begin{enumerate}
\item[(a)] ${\Bbb F}_2\cong V^M_0\subset V_1^M\subset\cdots\subset V_{n-2}^M\subset V_{n-1}^M=\mathcal{V}_{2r}\subset   V^M_n = {\Bbb F}^{2r}_2$,
where $V^M_0$ is generated by $\underline{1}=(1,\ldots,1)\in {\Bbb F}_2^{2r}$;
\item[(b)] For $d=\sum\limits_{i=0}^{n-1}i d_i<n$ with
each $d_i\geq0$, $v_{\omega_{d_0}}\circ\cdots\circ
v_{\omega_{d_{n-1}}}\in V^M_{d}$, where
$$v_{\omega_{d_i}}=v^{(i)}_1\circ\cdots\circ v^{(i)}_{d_i},\ v^{(i)}_j\in V^M_i.$$
 The operation $\circ$ on ${\Bbb F}^{2r}_2$ corresponds to
the cup product in $H^*_{G_\tau}(M;{\Bbb F}_2)$.
\end{enumerate}
\end{thm}

Each $V^M_k$ above can be thought of as a binary code in ${\Bbb F}_2^{2r}$.
Theorem~\ref{ring} and the Poincar\'e duality of $M$ implies that
 \begin{equation}\label{Equ:Sum-Dual}
   \dim_{{\Bbb F}_2} V^M_k  + \dim_{{\Bbb F}_2} V^M_{n-1-k} =
 \sum\limits_{j=0}^n b_j(M;\Bbb F_2)=2r.
 \end{equation}
  In addition, $V^M_{n-1-k}$ is perpendicular to $V^M_k$ with respect to
  $\langle \,\, , \, \rangle$. This is because
  $$H_G^k(M;{\Bbb F}_2)\cong V_k^Mt^k, \ \ H_G^{n-k-1}(M;{\Bbb F}_2)\cong V_{n-k-1}^Mt^{n-k-1}.$$
  So for any $x\in V_k^M$ and $y\in V_{n-k-1}^M$, we have
 $xt^k \cup yt^{n-k-1}=(x\circ y) t^{n-1}$ belongs to
 $H^{n-1}_G(M;{\Bbb F}_2)\cong \mathcal{V}_{2r}t^{n-1}$ by Theorem~\ref{ring}(b).
 Then by Lemma~\ref{code}, $x\circ y\in \mathcal{V}_{2r}$ implies
  $\langle x, y\rangle=0$. So we have $V^M_{n-1-k}\subset (V^M_k)^{\perp}$.
  Moreover, $\dim_{\mathbb{F}_2} V^M_{n-1-k} = \dim_{\mathbb{F}_2}  (V^M_k)^{\perp}$ by~\eqref{Equ:Sum-Dual}. This implies that
   \begin{equation} \label{Equ:Perp}
     (V^M_k)^{\perp} = V^M_{n-1-k}.
  \end{equation}

  \begin{cor}\label{gen-self-dual}
 $V_k^M \subset \mathbb{F}_2^{2r}$ is self-dual if and only if
  $\dim_{\mathbb{F}_2} V_k^M = \sum\limits_{j=0}^k b_j(M;\mathbb{F}_2)=r$.
  \end{cor}
  \begin{proof}
   The necessity is trivial. If $\dim_{\mathbb{F}_2} V_k^M=r$, then
    $\dim_{\mathbb{F}_2} V^M_{n-1-k} =r$ by~\eqref{Equ:Sum-Dual}.
    But by Theorem~\ref{ring}(a), we have either
    $V^M_k \subset V^M_{n-1-k}$ or $V^M_{n-1-k} \subset V^M_k$. Then
     $V^M_k$ and $V^M_{n-1-k}$ must be equal since they have the same dimension.
    So by~\eqref{Equ:Perp}, $(V^M_k)^{\perp} = V^M_{n-1-k} = V_k^M$. Hence
    $V_k^M$ is self-dual.
   \end{proof}
   
\vskip .6cm

  \section{Small covers with $\mathrm{m}$-involutions}

  \subsection{Small covers}\label{small}
  An $n$-dimensional simple (convex) polytope is a polytope such that each vertex of the polytope is exactly the intersection of $n$ \emph{facets} ($(n-1)$-dimensional faces) of the polytope.
   Following \cite{dj}, an $n$-dimensional {\em small cover} $\pi: M^n\rightarrow P^n$ is a
   closed smooth $n$-manifold $M^n$ with a locally
  standard $\Z_2^n$-action whose orbit space is homeomorphic to an $n$-dimensional
   simple convex polytope
  $P^n$, where a locally standard $\Z_2^n$-action on $M^n$ means that this
$\Z_2^n$-action on $M^n$ is locally isomorphic to a
faithful representation of $\Z_2^n$ on $\R^n$.
 Let $V(P^n)$ denote the set of all vertices of $P^n$ and $\mathcal{F}(P^n)$ denote
 the set of all facets of $P^n$.
 For any facet $F$ of $P^n$, the isotropy subgroup of
    $\pi^{-1}(F)$ in $M^n$ with respect to the $\Z_2^n$-action is
    a rank one subgroup of $\Z_2^n$ generated by an element of $\Z_2^n$,
    denoted by $\lambda(F)$.
    Then we obtain a map $\lambda: \mathcal{F}(P^n) \rightarrow
    \Z_2^n$ called
    the \emph{characteristic function} associated
     to $M^n$, which maps the $n$ facets meeting at each vertex of $P^n$
      to $n$ linearly independent elements in $\Z_2^n$.
     It is shown in~\cite{dj} that
     up to equivariant homeomorphisms, $M^n$
     can be recovered from $(P^n,\lambda)$ in a canonical way (see~\eqref{Equ:Glue-Back}).
  Moreover, many algebraic
 topological invariants of a small cover $\pi: M^n\rightarrow P^n$ can be easily
 computed from $(P^n, \lambda)$. Here is a list of facts on the cohomology rings of
 small covers proved in~\cite{dj}. \vskip .1cm

\begin{itemize}

\item[(R1)] Let $b_i(M^n;{\Bbb F_2})$ be the $i$-th mod 2 Betti number of $M^n$. Then
$$ b_i(M^n;{\Bbb F_2}) = h_i(P^n), \ 0\leq i \leq n$$
where $(h_0(P^n), h_1(P^n), \ldots, h_n(P^n))$ is the $h$-vector of $P^n$. \vskip .1cm

\item[(R2)] Let $M^{\Z_2^n}$ denote the fixed point set of the $\Z^n_2$-action on $M^n$. Then
 $$|M^{\Z_2^n}|=\sum_{i=0}^n b_i(M^n;{\Bbb F_2})=\sum_{i=0}^n h_i(P^n)=
   |V(P^n)|.$$ \vskip .1cm

\item[(R3)] The equivariant cohomology $H^*_{\Z_2^n}(M; {\Bbb F}_2)$ is isomorphic as graded rings to the Stanley--Reisner ring of $P^n$
 \begin{equation} \label{Equ:Equiv-Cohomology}
     H^*_{\Z_2^n}(M^n; {\Bbb F}_2) \cong {\Bbb F}_2(P^n)={\Bbb F}_2[a_{F_1},
        \ldots, a_{F_m}]/\mathcal{I}_{P^n}
     \end{equation}
where $F_1, \ldots, F_m$ are all the facets of $P^n$ and $a_{F_1},\ldots,a_{F_m}$ are of degree $1$, and $\mathcal{I}_{P^n}$ is the ideal generated by all square free monomials of $a_{F_{i_1}}\cdots a_{F_{i_s}}$ with $F_{i_1}\cap\cdots\cap F_{i_s}=\varnothing$ in $P^n$.\vskip .1cm

\item[(R4)] The mod-$2$ cohomology ring $H^*(M; {\Bbb F}_2)\cong {\Bbb F}_2[a_{F_1},
 \ldots, a_{F_m}]/\mathcal{I}_P+J_\lambda$, where $J_\lambda$ is an ideal determined by $\lambda$. In particular,
 $H^*(M; {\Bbb F}_2)$ is generated by degree $1$ elements.

\end{itemize}

\subsection{Spaces constructed from simple polytopes with $\Z_2^r$-colorings} \ \vskip .1cm
  Let $P^n$ be an $n$-dimensional simple polytope in $\R^n$.
    For any $r\geq 0$, a
  $\Z_2^r$-coloring on $P^n$ is a map
  $\mu: \mathcal{F}(P^n) \rightarrow \Z_2^r$. For any facet $F$ of $P^n$,
  $\mu(F)$ is called the \emph{color of $F$}.
  Let $f=F_1\cap \cdots \cap F_k$ be a codimension-$k$ face of $P^n$ where
  $F_1,\ldots, F_k\in \mathcal{F}(P^n)$. Define
     \begin{equation}\label{Equ:Subgroup}
     G^{\mu}_f = \text{the subgroup of $\Z_2^r$ generated by}\
     \mu(F_1), \ldots , \mu(F_k).
   \end{equation}
   Besides, let $G^{\mu}$ be the subgroup of $\Z_2^r$ generated by
    $\{ \mu(F)\,;\, F\in \mathcal{F}(P^n) \}$.
   The rank of $G^{\mu}$ is called the \emph{rank of $\mu$}, denoted by $\mathrm{rank}(\mu)$.
   It is clear that
   $\mathrm{rank}(\mu) \leq r$.\vskip .1cm

    For any point $p\in P^n$, let $f(p)$ denote the unique face of $P^n$ that contains $p$ in
  its relative interior.
   Then we define a space associated to $(P^n,\mu)$ by:
    \begin{equation} \label{Equ:Glue-Back}
        M(P^n,\mu) = P^n\times \Z_2^r \slash \sim
     \end{equation}
   where $(p,g) \sim (p',g')$ if and only if $p=p'$ and
   $g^{-1}g' \in G^{\mu}_{f(p)}$.

   \begin{itemize}
     \item $M(P^n,\mu)$ is a closed manifold if
        $\mu$ is \emph{non-degenerate} (i.e. $\mu(F_1),\ldots, \mu(F_k)$ are 
        linearly independent whenever $F_1\cap\cdots \cap F_k\neq \varnothing$). \vskip .1cm
   \item $M(P^n,\mu)$ has $2^{r-\mathrm{rank}(\mu)}$ connected components. So $M(P^n,\mu)$
    is connected if and only if
   $\mathrm{rank}(\mu)=r$. \vskip .1cm
   \item There is a canonical $\Z_2^r$-action on $M(P^n,\mu)$ defined by:
     $$h\cdot [(x,g)] = [(x,g+h)],\ x\in P^n,\, g,h\in \Z_2^r.$$
     let
    $\pi_{\mu}: M(P^n,\mu)\rightarrow P^n $ be the map sending any $[(x,g)]\in  M(P^n,\mu)$ to
    $x\in P^n$.
   \end{itemize}
   \vskip .1cm

    For any face $f$ of $P^n$ with $\dim(f)\geq 1$, let
     $r(f) = r-\mathrm{rank}(G^{\mu}_f)$ and
    $$\eta_f: \Z_2^r \rightarrow \Z_2^r\slash G^{\mu}_f \cong \Z_2^{r(f)}$$ be
    the quotient homomorphism.
    Then $\mu$ induces a $\Z_2^{r(f)}$-coloring $\mu_f$ on $f$ by:
       \begin{equation}\label{Induced-Coloring}
        \mu_f(F\cap f) := \eta_f( \mu(F)), \ \text{where} \ F\in \mathcal{F}(P^n), \
          \dim(F\cap f)=\dim(f)-1.
       \end{equation}
    It is easy to see that $\pi^{-1}_{\mu}(f)$ is homeomorphic to $M(f, \mu_f)$. \vskip .1cm

   \begin{exam} \label{Exam:Small-Cover}
   Suppose $\pi: M^n \rightarrow P^n$ is a small cover
    with characteristic function $\lambda$. Then
    $M^n$ is homeomorphic to $M(P^n,\lambda)$. For any face $f$ of $P^n$,
     $\pi^{-1}(f)\cong M(f, \lambda_f)$ is a closed connected
     submanifold of $M^n$ (called a \emph{facial submanifold} of $M^n$), which is a small cover over $f$.
    \end{exam}
    \vskip .1cm

   \subsection{Small covers with regular $\mathbf{m}$-involutions}\label{involution}
  Let $\pi: M^n\rightarrow P^n$ be a small cover over an $n$-dimensional simple polytope $P^n$
  and $\lambda: \mathcal{F}(P^n)\rightarrow \Z_2^n$ be its characteristic
function. Let us discuss under what condition there exists a regular $\mathrm{m}$-involution
on $M^n$.

\begin{thm}\label{max-involution}
The following statements are equivalent.
\begin{itemize}
\item[(a)] There exists a regular $\mathrm{m}$-involution on $M^n$.
\item[(b)] There exists a regular involution on $M^n$ with only isolated fixed points;
\item[(c)] The image $\IM \lambda \subset \Z_2^n$ of $\lambda$
   consists of exactly $n$ elemnets (which implies that $P^n$ is
 $n$-colorable) and so they form a basis of $\Z_2^n$.
\end{itemize}
\end{thm}
\begin{proof}
 (a) implies (b) since by definition an $\mathrm{m}$-involution only
 has isolated fixed point.  \vskip .1cm

 (b)$\Rightarrow$(c) Suppose there exists $g\in \Z_2^n$ so that the fixed points of $\tau_g$
  on $M^n$ are all isolated.
  Let $v$ be an arbitrary vertex on $P^n$ and $F_1,\ldots,F_n$ be the $n$ facets
  meeting at $v$. By the construction of small covers,
  $\pi^{-1}(v)=p$ is a fixed point of the whole group $\Z_2^n$.
  Let $U\subset M$ be a small neighborhood of $p$. Since the action of $\Z_2^n$ on $M^n$ is locally standard, we observe that for
  $h= \lambda(F_{i_1}) + \cdots + \lambda(F_{i_s}) \in \Z_2^n$, $1\leq i_1 < \cdots < i_s\leq n$,
  the dimension of the fixed point set of $\tau_h$ in $U$ is equal to $n-s$.
  Then since the fixed points of $\tau_g$ are all isolated, we must have
   $g=\lambda(F_1)+\cdots +
\lambda(F_n)$.\vskip .1cm
  Next, take an edge of $P^n$ with two endpoints $v_1, v_2$.
  Since $P^n$ is simple, there are $n+1$ facets $F_1, \ldots, F_n, F'_n$ such that
 $v_1=F_1\cap\cdots\cap F_{n-1}\cap F_n$ and $v_2=F_1\cap\cdots\cap
F_{n-1}\cap F'_{n}$. Then
$\lambda(F_1)+\cdots+\lambda(F_{n-1})+\lambda(F_n)=g=\lambda(F_1)+\cdots+\lambda(F_{n-1})+\lambda(F'_{n})$,
which implies $\lambda(F_n)=\lambda(F'_{n})$.
Since the 1-skeleton of $P^n$ is connected,
we can deduce the image $\IM \lambda$ of $\lambda$ consists of $n$ elements of $\Z_2^n$
 which form a basis of $\Z_2^n$.\vskip .1cm

 (c)$\Rightarrow$(a) Suppose $\IM \lambda =\{ g_1,\ldots, g_n\}$
 is a basis of $\Z_2^n$. Then by the construction of small covers, the fixed point set of
  the regular involution $\tau_{g_1+\cdots+g_n}$ on $M^n$ is
  $$\{ \pi^{-1}(v)\,|\, v\in V(P^n)\} = M^{\Z_2^n}.$$
 So the number of fixed points of $\tau_{g_1+\cdots+g_n}$ is equal to the number of
 vertices of $P^n$, which is known to be $h_0(P^n)+h_1(P^n)+\cdots+h_n(P^n)$.
 Then by the result (R1) in section~\ref{small}, $\tau_{g_1+\cdots+g_n}$ is
 an $\mathrm{m}$-involution on $M^n$.
\end{proof}

\begin{rem}
It should be pointed out that for an $n$-colorable simple $n$-polytope $P^n$,
 the image of a characteristic
function $\lambda: \mathcal{F}(P^n)\rightarrow \Z_2^n$ might consist of more than
 $n$ elements of $\Z_2^n$. In that case,
the small cover defined by $P^n$ and $\lambda$ admits no regular $\mathrm{m}$-involutions.
So Theorem~\ref{max-involution} only tells us that
if an $n$-dimensional small cover $M^n$ over $P^n$ admits a regular $\mathrm{m}$-involution, then $P^n$ is $n$-colorable. But conversely, this is not true.
\end{rem}

\subsection{Descriptions of $n$-colorable $n$-dimensional simple polytopes}
\ \vskip .1cm

  The following descriptions of $n$-colorable simple $n$-polytopes are due to Joswig~\cite{jos}.

 \begin{thm} [{\cite[Theorem 16 and Corollary 21]{jos}}]\label{j}
 Let $P^n$ be an $n$-dimensional simple polytope, $n\geq 3$.
  The following statements are equivalent.
 \begin{itemize}
   \item[(a)]$P^n$ is $n$-colorable;\vskip .1cm
   \item[(b)] Each $2$-face of $P^n$ has an even number of vertices.\vskip .1cm
   \item[(c)] Each face of $P^n$ with dimension greater than $0$ (including $P^n$ itself) has an even number of vertices.
   \item[(d)] Any proper $k$-face of $P^n$ is $k$-colorable.
 \end{itemize}
 \end{thm}
 Later we will give some
  new descriptions of $n$-colorable simple $n$-polytopes from our study of
 binary codes associated to general simple polytopes in section 5.
 \vskip .6cm

\section{Self-dual binary codes from small covers}
 Let $\pi: M^n\rightarrow P^n$ be an $n$-dimensional small cover which admits a regular $\mathrm{m}$-involution.
 By Theorem~\ref{max-involution}, $P^n$ is an $n$-dimensional $n$-colorable simple polytope
 with an even number of vertices. Let
  $\{ v_1,\ldots, v_{2r} \}$ be all the vertices of $P^n$.
 The characteristic function $\lambda$ of $M^n$
 satisfies: $\mathrm{Im}(\lambda) =\{e_1,\ldots, e_n\}$ is a basis of $\Z_2^n$.
By Theorem~\ref{max-involution}, $\tau_{e_1+\cdots+e_n}$ is an $\mathrm{m}$-involution
on $M^n$. So by the discussion in section~\ref{general-case}, we obtain a filtration
$${\Bbb F}_2\cong V^{M}_0\subset V_1^{M}\subset\cdots\subset V_{n-2}^{M}\subset V_{n-1}^{M}=\mathcal{V}_{2r}\subset V^M_n= {\Bbb F}^{2r}_2.$$
According to Theorem~\ref{ring} and the property (R1) of small covers,
 $$\dim_{{\Bbb F}_2} V^M_k=\sum\limits_{j=0}^k b_j(M^n; {\Bbb F}_2)= \sum\limits_{j=0}^k h_j(P^n)
  , \ 0\leq k \leq n.$$
 Then since $h_j(P^n)>0$ for all $0\leq j\leq n$, we have
 $V^M_0\subsetneq V_1^M\subsetneq\cdots\subsetneq V_{n-1}^M\subsetneq  V^M_n = {\Bbb F}^{2r}_2$.
  Note that $V^M_k$ is self-dual in ${\Bbb F}^{2r}_2$ if and only if
 $V^M_k = (V^M_k)^{\perp} = V^M_{n-1-k}$ by~\eqref{Equ:Perp}.
 Then $V^M_k$ is self-dual if and only if $k=n-1-k$ (i.e. $n$ is odd and $k=\frac{n-1}{2}$).
 So we prove the following proposition.

\begin{prop}\label{suff-nece}
Let $\pi: M^n\rightarrow P^n$ be an $n$-dimensional small cover which admits a regular $\mathrm{m}$-involution. Then $V_k^{M}$ is a self-dual code if and only if
$n$ is odd and $k={{n-1}\over 2}$.
\end{prop}

 In the remaining part of this section, we will describe each $V_k^{M}$, $0 \leq k \leq n$, explicitly in terms of the combinatorics of $P^n$.
 First, any face $f$ of $P^n$ determines an element $\xi_f\in {\Bbb F}^{2r}_2$ where
  the $i$-th entry of $\xi_f$ is $1$ if and only if $v_i$ is a vertex of $f$.
 In particular, $\xi_{P^n} = \underline{1} = (1,\ldots , 1)\in {\Bbb F}^{2r}_2 $ and
 $\{\xi_{v_1},\cdots, \xi_{v_{2r}}\}$ is a linear basis of $\mathbb{F}^{2r}$.
 Note that for any faces $f_1,\ldots, f_s$ of $P^n$, we have
  \begin{equation} \label{Equ:Product-Face}
     \xi_{f_1\cap\cdots\cap f_s} = \xi_{f_1}\circ\cdots\circ \xi_{f_s}.
  \end{equation}
 We define a sequence of binary codes $\mathfrak{B}_k(P^n) \subset {\Bbb F}^{2r}_2$
 as follows.
 \begin{equation} \label{Equ:Def-Bk}
    \mathfrak{B}_k(P^n):= \text{Span}_{\Bbb F_2}\{\xi_f \,;\,  f  \
                                                 \text{is a codimension-$k$ face of
      $P^n$}\},\ 0\leq k \leq n.
    \end{equation}
    \vskip .1cm

\begin{rem}
  Changing the ordering of the vertices of $P^n$ only causes the coordinate changes in
  $\mathbb{F}_2^n$. So up to equivalences of binary codes, each
   $\mathfrak{B}_k(P^n)$ is uniquely determined by $P^n$.
\end{rem}

 \begin{lem} \label{Lem:Inclusion}
 For any $n$-colorable simple $n$-polytope $P^n$ with $2r$ vertices, we have
   $$ \mathfrak{B}_0(P^n) \subset \mathfrak{B}_1(P^n) \subset\cdots\subset
   \mathfrak{B}_{n-1}(P^n) = \mathcal{V}_{2r} \subset \mathfrak{B}_n(P^n)\cong {\Bbb F}_2^{2r}.$$
 \end{lem}
 \begin{proof}
  By definition, $P^n$ can be colored by $n$ colors $\{e_1,\ldots, e_n\}$. Now choose an arbitrary color say $e_j$, we observe that
    each vertex of $P^n$ is contained in exactly one facet of $P^n$ colored by $e_j$.
    This implies that
    $$ \xi_{P^n} = \xi_{F_1} +\cdots + \xi_{F_s} $$
    where $F_1,\ldots, F_s$ are all the facets of $P^n$ colored by $e_j$.
    So $\mathfrak{B}_0(P^n) \subset \mathfrak{B}_1(P^n)$.
   Moreover, by Theorem~\ref{j}(d), the facets $F_1,\ldots, F_s$ are $(n-1)$-dimensional simple polytopes which are $(n-1)$-colorable. So by repeating the above argument, we can
   show that $\mathfrak{B}_1(P^n) \subset \mathfrak{B}_2(P^n)$ and so on. Now it remains to show
   $\mathfrak{B}_{n-1}(P^n) = \mathcal{V}_{2r}$. \vskip .1cm

   By definition, $\mathfrak{B}_{n-1}(P^n)$ is spanned by $\{ \xi_f \, |\, f\ \text{is an edge (or $1$-face) of}\ P^n \}$. So it is obvious that $\mathfrak{B}_{n-1}(P^n) \subset \mathcal{V}_{2r}$.
  Let $\{ v_1,\ldots, v_{2r} \}$ be all the vertices of $P^n$. It is easy to see that
  $\mathcal{V}_{2r}$ is spanned by $\{ \xi_{v_i} + \xi_{v_j} \,|\, 1\leq i\neq j \leq 2r\}$.
  Then since there exists an edge path on $P^n$ between any two vertices $v_i$ and $v_j$ of
  $P^n$, $\xi_{v_i} + \xi_{v_j}$ belongs to $\mathfrak{B}_{n-1}(P^n)$. So
  $\mathcal{V}_{2r}\subset \mathfrak{B}_{n-1}(P^n)$. This finishes the proof.
    \end{proof}
 \vskip .1cm

 Later we will prove that
 the condition in Lemma~\ref{Lem:Inclusion} is also sufficient for an $n$-dimensional
 simple polytope
 to be $n$-colorable (see Proposition~\ref{collection}).\vskip .1cm

 \begin{thm} \label{Thm:Main-1}
  Let $\pi: M^n\rightarrow P^n$ be an $n$-dimensional small cover which admits a regular $\mathrm{m}$-involution.
  For any $0\leq k \leq n$, the space $V^{M}_k$ coincides with $\mathfrak{B}_k(P^n)$.
 \end{thm}
  \vskip .1cm

{\begin{cor}\label{col-self}
Let $P^n$ be an $n$-colorable simple $n$-polytope with $2r$ vertices. Then
  $$\dim_{{\Bbb F}_2}\mathfrak{B}_k(P^n)=
  \sum_{i=0}^k h_i(P^n), \  0\leq k\leq n.$$
 If $n$ is odd, then
$\mathfrak{B}_k(P^n)$ is a self-dual code in $\mathbb{F}^{2r}_2$
if and only if $k=\frac{n-1}{2}$.
If $n$ is even, $\mathfrak{B}_k(P^n)$ cannot be a self-dual code in $\mathbb{F}^{2r}_2$ for
any $0\leq k \leq n$.
\end{cor}
\begin{proof}
Let $M^n$ be a small cover over $P^n$ whose characteristic function
 $\lambda: \mathcal{F}(P^n)\rightarrow \Z_2^n$ satisfies:
the image $\mathrm{Im} (\lambda)$  is a basis $\{e_1,\ldots,e_n\}$ in $\Z_2^n$.
Then by Theorem~\ref{Thm:Main-1}, $\mathfrak{B}_k(P^n)$ coincides with
$V^M_k$. So this corollary follows from Theorem~\ref{ring} and
Proposition~\ref{suff-nece}.
 \end{proof}

 \begin{cor}  \label{Cor:Main-2}
    Let $\pi: M^n\rightarrow P^n$ be an $n$-dimensional small cover which admits a regular $\mathrm{m}$-involution where
      $n$ is odd. Then the self-dual binary code $C_{M^n} = V^{M}_{\frac{n-1}{2}} =
    \mathfrak{B}_{\frac{n-1}{2}}(P^n)$ is spanned by
    $\{ \xi_{f}\,;\, f\  \text{is any face of $P^n$ with}\  \dim(f) = \frac{n+1}{2} \}$.
    So the minimum distance of $C_{M^n}$ is less or equal to
     $\mathrm{min}\{ \# (\text{vertices of}\ f)  \, ;\, f \  \text{is
  a}\ \frac{n+1}{2}   \text{-dimensional face of}\ P^n\}$.
 \end{cor}

 \noindent \textbf{Problem 1:} For any $n$-dimensional small cover $M^n$
 which admits a regular $\mathrm{m}$-involution where $n$ is odd, determine
 the minimum distance of the self-dual binary code $C_{M^n}$. \vskip .2cm

 We will see in Proposition~\ref{prop:3-polytope} that
 when $n=3$, the minimum distance of
 $C_{M^n}$ is always equal to $4$. For higher dimensions,
 it seems to us that the minimum distance of $C_{M^n}$ should be equal to
 $\mathrm{min}\{ \# (\text{vertices of}\ f)  \, ;\, f \ \text{is
  a}\ \frac{n+1}{2}\text{-dimensional face of}\ P^n\}$. But the proof is not clear to us.

 \vskip .2cm

 In the following, we are going to prove Theorem~\ref{Thm:Main-1}.
  For brevity, let
   $$\tau=\tau_{e_1+\cdots+e_n}, \ \ \
   G_{\tau}=\langle e_1+\cdots+e_n \rangle\cong \Z_2 \subset \Z_2^n.$$
    By the construction of $M^n$,
  all the fixed points of $\tau$ on $M^n$ are
  $\tilde{v}_1,\ldots, \tilde{v}_{2r}$ where
   $$\tilde{v}_i=\pi^{-1}(v_i) \in M^n, \ i=1,\ldots, 2r.$$
 \vskip .2cm

\subsection{Proof of Theorem~\ref{Thm:Main-1}} \ \vskip .1cm

  According to the result in (R4) of section 3.1,
  the cohomology ring $H^*(M^n; {\Bbb F}_2)$ of $M^n$
   is generated as an algebra by $H^1(M^n;{\Bbb F}_2)$.
  So as an algebra over $H^*(BG_\tau; {\Bbb F}_2)={\Bbb F}_2[t]$,
   the equivariant cohomology ring
   $H^*_{G_\tau}(M^n; {\Bbb F}_2)=H^*(M^n;{\Bbb F}_2)\otimes H^*(BG_\tau; {\Bbb F}_2)$ is
   generated by elements of degree $1$. In addition, 
   the operation $\circ$ on ${\Bbb F}^{2r}_2$ corresponds to the cup product in $H^*_{G_\tau}(M^n;{\Bbb F}_2)$. So
  we obtain from Theorem~\ref{ring} that for any $1\leq k \leq n$,
   $V^M_k =  \underset{k}{\underbrace{V^M_1\circ\cdots \circ V^M_1}}$.
   On the other hand, there is a similar structure on $\mathfrak{B}_k(P^n)$ as well.
   \vskip .1cm

 \textbf{Claim-1:}
   $\mathfrak{B}_k(P^n) =  \underset{k}{\underbrace{\mathfrak{B}_1(P^n)
   \circ\cdots\circ \mathfrak{B}_1(P^n)}}$, $1\leq k \leq n$.

   Indeed, for any $k$ different facets $F_{i_1},\ldots, F_{i_k}$ of $P^n$,
   their intersection $F_{i_1}\cap\cdots\cap F_{i_k}$
   is either empty or a face of codimension $k$. So by~\eqref{Equ:Product-Face}, we have
   $\xi_{F_{i_1}}\circ\cdots\circ \xi_{F_{i_k}} = \xi_{F_{i_1}\cap\cdots\cap F_{i_k}} \in \mathfrak{B}_{k}(P^n)$.
   If there are repetitions of facets in $F_{i_1},\ldots, F_{i_k}$, we have
   $\xi_{F_{i_1}}\circ\cdots\circ \xi_{F_{i_k}}
    \in \mathfrak{B}_{l}(P^n)$ for some $l<k$ (because $x\circ x =x$ for any
    $x\in {\Bbb F}_2^{2r}$).
    But since
    $P^n$ is $n$-colorable in our case, we have
    $\mathfrak{B}_{l}(P^n) \subset \mathfrak{B}_{k}(P^n)$ by
    Lemma~\ref{Lem:Inclusion}.
   Conversely,
    any codimension-$k$ face $f$ of $P^n$ can be written as
 $f=F_{i_1}\cap\cdots\cap F_{i_k}$ where $F_{i_1}, \ldots, F_{i_k}$ are $k$
 different facets of $P^n$.
  So
  $\xi_f=\xi_{F_{i_1}\cap\cdots\cap F_{i_k}}=\xi_{F_{i_1}}\circ\cdots\circ \xi_{F_{i_k}}$.
  The Claim-1 is proved.\vskip .1cm

  So to prove Theorem~\ref{Thm:Main-1}, it is sufficient to prove that
   $V^M_1 = \mathfrak{B}_1(P^n)$, i.e., $V^M_1$ is spanned by the set
   $\{ \xi_F\,;\, F\  \text{is any facet of $P^n$} \}$.
   Next, we examine the localization of $H^1_{\Z_2^n}(M^n; \mathbb{F}_2)$ to
 $H^1_{\Z_2^n}(M^{\Z_2^n}; \mathbb{F}_2)$ more carefully.\vskip .1cm

  Let $\mathcal{F}(P^n)=\{F_1, \cdots, F_m\}$ be the set of all facets of $P^n$.
  By our previous notations,
   the regular involution $\tau = \tau_{e_1+\cdots+e_n}$ on $M^n$ only has isolated fixed points:
    $$M^{G_{\tau}}=M^{\Z_2^n} =\{ \tilde{v}_1,\ldots, \tilde{v}_{2r} \}.$$

 Clearly the inclusion $G_\tau\hookrightarrow \Z_2^n$ induces the diagonal maps
 $\Delta_E: EG_\tau\longrightarrow E\Z_2^n$ and $\Delta_B: BG_\tau\longrightarrow B\Z_2^n$ such that the following diagram commutes
 $$\CD
  EG_\tau  @>\Delta_E >> E\Z_2^n \\
  @V  VV @V  VV  \\
  BG_\tau  @>\Delta_B >> B\Z_2^n.
\endCD$$
Since $M^{G_\tau}=M^{\Z_2^n}$ consists of isolated points,
 we have a commutative diagram
$$\xymatrix{
  EG_\tau\times M^{G_\tau} \ar[rrr]^{\Delta_E\times\text{id}}
  \ar[dr]^{i_1} \ar[ddd]_{} & & &
  E\Z_2^n\times M^{\Z_2^n} \ar[dl]_{i_2} \ar[ddd]^{ }\\
  & EG_\tau\times M^n \ar[r]^{\Delta_E\times \text{id} }
  \ar[d]_{}
      & E\Z_2^n\times M^n \ar[d]^{ } &      \\
  & EG_\tau\times_{G_\tau} M^n \ar[r]^{\phi }&
    E\Z_2^n\times_{\Z_2^n} M^n &  \\
  EG_\tau\times_{G_\tau} M^{G_\tau}\ar[ur]^{i_3} \ar[rrr]_{\Delta_B\times\text{id}=\psi} & &  &
  E\Z_2^n\times_{\Z_2^n} M^{\Z_2^n}\ar[ul]_{i_4}        }$$
where  $\phi$ is the map
induced by $\Delta_E\times \text{id}$, and $i_1, i_2, i_3, i_4$ are all inclusions. Furthermore, we have the following commutative diagram where $i_3^*\circ\phi^*=\psi^*\circ i_4^*$.
\begin{equation}\label{graph}
\xymatrix{
  H^*_{\Z_2^n}(M^n;{\Bbb F}_2) \ar[r]^{\  \phi^*}\ar[d]_{i_4^*} &
  H^*_{G_\tau}(M^n; {\Bbb F}_2)\ar[d]^{i_3^*}\\
  H^*_{\Z_2^n}(M^{\Z_2^n};{\Bbb F}_2) \ar[r]_{\ \psi^*} &
  H^*_{G_\tau}(M^{G_\tau};{\Bbb F}_2)
}
\end{equation}
  Note that $i_3^*$ and $i_4^*$ are injective, and
 $$H^*_{\Z_2^n}(M^{\Z_2^n};{\Bbb F}_2)
\cong \bigoplus_{v\in V(P^n)}H^*_{\Z_2^n}(\tilde{v};{\Bbb F}_2),\ \
 H^*_{G_\tau}(M^{G_\tau};{\Bbb F}_2)
\cong \bigoplus_{ v\in V(P^n)} H^*_{G_\tau}(\tilde{v};{\Bbb F}_2),$$
where $\tilde{v}=\pi^{-1}(v)$ is the fixed point corresponding to a vertex $v\in V(P^n)$.
Then by the fact that
 $H^*_{\Z_2^n}(\tilde{v};{\Bbb F}_2)\cong H^*(B\Z_2^n;{\Bbb F}_2)$ and $H^*_{G_\tau}(\tilde{v};{\Bbb F}_2)\cong H^*(BG_\tau;{\Bbb F}_2)$, we
 can regard $\psi^*$ as a direct sum:
  \begin{equation} \label{Equ:Psi}
     \psi^* = \bigoplus_{v\in V(P^n)}\Delta_B^*.
  \end{equation}
We know that $H^*(B\Z_2^n;{\Bbb F}_2)={\Bbb F}_2[t_1, \ldots, t_n]$ with $\deg t_i=1$, and $H^*(BG_\tau;{\Bbb F}_2)={\Bbb F}_2[t]$ with $\deg t=1$. For each $1\leq i \leq n$, let
$G_i=\langle e_i\rangle\cong \Z_2 \subset \Z^n_2$.
 It is clear that
 $$
  H^*(BG_i; {\Bbb F}_2)={\Bbb F}_2[t_i],\, 1\leq i \leq n;\ \
  \Z^n_2 = G_1\times \cdots \times G_n.  $$
 For any $1\leq i \leq n$, let $\zeta_i: G_i\rightarrow G_{\tau}$ be the
  group isomorphism sending $e_i\rightarrow e_1+\cdots +e_n  $, and
 let $\rho_i: \Z^n_2\rightarrow G_i$ be the projection sending $e_j$ to $0$ for any
  $1\leq j\neq i \leq n$. Let $\theta: G_{\tau} \hookrightarrow \Z_2^n$
   be the inclusion map.
  It is clear that
   $$ \rho_i\circ \theta \circ \zeta_i =\mathrm{id}_{G_i}, \ 1\leq i \leq n.$$

  Let $B_{\zeta_i}: BG_i\rightarrow BG_{\tau}$ and $B_{\rho_i}:B\Z^n_2\rightarrow BG_i$
    be the maps
   induced by $\zeta_i$ and $\rho_i$ between the classifying spaces, respectively.
    Then since
  there is a functorial construction of classifying spaces of groups (see~\cite{milnor}),
   we can assume
    $B_{\rho_i}\circ\Delta_B\circ B_{\zeta_i} = \mathrm{id}_{BG_i}$
     (recall that $\Delta_B: BG_\tau\rightarrow B\Z_2^n$ is induced by $\theta$).
     So for any $1\leq i \leq n$, we have
  $$  \mathrm{id}^*_{BG_i} =  B^*_{\zeta_i}\circ \Delta^*_B \circ B^*_{\rho_i}:
     H^1(BG_i;{\Bbb F}_2)\rightarrow
      H^1(B\Z^n_2;{\Bbb F}_2) \rightarrow H^1(BG_{\tau};{\Bbb F}_2)
   \rightarrow H^1(BG_i;{\Bbb F}_2).$$
   Obviously, we have $B^*_{\rho_i}(t_i)=t_i$
    for any $1\leq i \leq n$.
    In addition, we can assert $B^*_{\zeta_i}(t)=t_i$ since $B^*_{\zeta_i}$ is
   an isomorphism, and $t$ and $t_i$ are the unique generators
   of $H^1(BG_{\tau};{\Bbb F}_2)$ and $H^1(BG_i;{\Bbb F}_2)$, respectively.
   Then $t_i=B^*_{\zeta_i}\circ \Delta^*_B \circ B^*_{\rho_i}(t_i) =
   B^*_{\zeta_i}\circ \Delta^*_B(t_i)$ implies
      \begin{equation} \label{Equ:t-ti}
         \Delta_B^*(t_i)=t, \ 1\leq i \leq n.
      \end{equation}
       \vskip .1cm

    Our strategy here is to understand the image of the localization map $i^*_3$ in terms of
       $\psi^*$ and $i_4^*$. So we need to show that $\phi^*$ is surjective. \vskip .1cm

  \textbf{Claim-2:} The homomorphism $\phi^*$ is surjective.\vskip .1cm

    Indeed, according to~\cite[Theorem 4.12]{dj}, the $E_2$-term of the Serre spectral sequence
    of the fibration $E\Z_2^n \times_{\Z_2^n} M^n \rightarrow B\Z^n_2$ collapses and we have
    $$H^*_{\Z_2^n}(M^n;{\Bbb F}_2)\cong  H^*(M^n;{\Bbb F}_2)\otimes H^*(B\Z_2^n;{\Bbb F}_2).$$
    It means that the small cover $M^n$ is \emph{equivariantly formal} (see~\cite{gkm} for the definition).
     Meanwhile, we already know that
    $H^*_{G_\tau}(M^n; {\Bbb F}_2)=H^*(M^n;{\Bbb F}_2)\otimes H^*(BG_\tau; {\Bbb F}_2)$. So
    the surjectivity of $\phi^*$ follows from the surjectivity of
    $\Delta_B^*: H^*(B\Z_2^n;{\Bbb F}_2)\rightarrow H^*(BG_\tau; {\Bbb F}_2)$
    which is implied by~\eqref{Equ:t-ti}. The Claim-2 is proved.\vskip .1cm

   \begin{rem} The surjectivity
of restriction map to the equivariant cohomology with respect to a
subgroup is known for many equivariant formal situations. For example, an explicit
statement of the surjectivity result in case of real torus actions is contained
in~\cite[Theorem 5.7]{amp}.
 \end{rem}

   By the Claim-2, the image of the localization
   $ i^*_3: H^*_{G_\tau}(M^n; {\Bbb F}_2) \rightarrow
    H^*_{G_\tau}(M^{G_\tau};{\Bbb F}_2)$ is
   \begin{equation}\label{Equ:Ident}
    \mathrm{Im}(i^*_3) = \mathrm{Im}(i^*_3\circ \phi^*)= \mathrm{Im}(\psi^*\circ i_4^*).
  \end{equation}

 For any fixed point $\tilde{v}\in M^{\Z_2^n}$, the inclusion $i_{\tilde{v}}:\{\tilde{v}\}\hookrightarrow M$ induces a homomorphism
  $$i_{\tilde{v}}^*: H^*_{\Z_2^n}(M;{\Bbb F}_2)\cong {\Bbb F}_2[a_{F_1}, \ldots, a_{F_m}]/I\longrightarrow H^*_{\Z_2^n}(\{\tilde{v}\};{\Bbb F}_2) \cong H^*(B\Z_2^n;{\Bbb F}_2)=
  {\Bbb F}_2[t_1, \ldots, t_n].$$
 Then we can write
  \begin{equation} \label{Equ:i4}
    i_4^*=\bigoplus_{v\in V(P^n)}i^*_{\tilde{v}}.
  \end{equation}

  Since we have already known how to compute $\psi^*$ from~\eqref{Equ:Psi} and~\eqref{Equ:t-ti},
    it remains to understand each $i^*_{\tilde{v}}$ for us to
    compute $\mathrm{Im}(i^*_3)$. This is given in the following lemma.
   \vskip .1cm

\begin{lem} \label{Lem:Local}
 Let $\lambda$ be the characteristic function of the small cover $M^n$
 so that $\mathrm{Im}(\lambda) =\{e_1,\ldots, e_n\}$ is a basis of $\Z_2^n$.
Suppose $F$ is a facet of $P^n$ with $\lambda(F) = e_{j}$ for some $1\leq j\leq n$.
 Then for any vertex $v$ of $P^n$, the fixed point $\tilde{v} = \pi^{-1}(v) \in M^{\Z_2^n}$  satisfies:
$$i_{\tilde{v}}^*(a_F)= \left\{\begin{array}{ll}
   t_{j}, &\text{ if } v \in F  ;\\
    0,           &\text{ if } v\notin F.
                              \end{array}
\right.
$$

\end{lem}

\begin{proof}
 Let $M_F=\pi^{-1}(F)$ and let $M_{\Z_2^n} = E\Z_2^n \times_{\Z_2^n} M^n$ and
 $(M_F)_{\Z_2^n} = E\Z_2^n \times_{\Z_2^n} M_F$ be the Borel constructions of
  $M^n$ and $M_F$, respectively.
  According to the discussion in~\cite[Section 6.1]{dj}, $a_F$ is the first Stiefel-Whitney class
  $w_1(L_F)$ of a line bundle $L_F$ over $M_{\Z_2^n}$. Moreover,
     the restriction of $L_{F}$ to $M_{\Z_2^n} - (M_{F})_{\Z_2^n}$ is a trivial line bundle.\vskip .1cm

For any fixed point $\tilde{v}$ of $M^n$, let $L_{\tilde{v}}$ denote the restriction of
the line bundle $L_F$ to the Borel construction $\{\tilde{v}\}_{\Z_2^n}=
 E\Z_2^n \times_{\Z_2^n} \{ \tilde{v} \}$.
  \vskip .1cm

  If a vertex $v \notin F$, so $\tilde{v} \notin M_F$ and then $L_{\tilde{v}}$
  is a trivial line bundle over $\{\tilde{v}\}_{\Z_2^n}$. So we have
     $$0=w_1(L_{\tilde{v}}) = i_{\tilde{v}}^*(w_1(L_F))= i_{\tilde{v}}^*(a_F).$$

 For any vertex $v\in F$, let $\mathfrak{p}_{\tilde{v}}: M^n \rightarrow \{\tilde{v}\}$ be
 the constant map.
It is clear that $\mathfrak{p}_{\tilde{v}} \circ i_{\tilde{v}} =
    id_{\{\tilde{v}\}}$. The induced maps $\mathfrak{p}^*_{\tilde{v}}$ and
    $i_{\tilde{v}}^*$ on the equivariant cohomology give
\begin{equation} \label{Equ:Composite}
  \xymatrix{\text{id}: \
  H^1_{\Z_2^n}(\{ \tilde{v} \})\ar[r]^{\mathfrak{p}^*_{\tilde{v}}} \ar@{=}[d] &
  H^1_{\Z_2^n}(M^n)\ar[r]^{i_{\tilde{v}}^*} \ar@{=}[d]&
  H^1_{\Z_2^n}(\{ \tilde{v} \})\ar@{=}[d]\\
 \text{id}: \  \text{span}\{t_1,\ldots,t_n\}\ar[r]^{\mathfrak{p}^*_{\tilde{v}}} &
 \text{span}\{a_{F_1},\ldots, a_{F_m}\}\ar[r]^{\ \; i_{\tilde{v}}^*} &
  \text{span}\{t_1,\ldots,t_n\} }
 \end{equation}

Let $\lambda$ be the characteristic function
 of the small cover $\pi: M^n \rightarrow P^n$. We can regard
 $\lambda$ as a linear map $\lambda:\Z_2^m=\text{span}\{F_1,\ldots,F_m\}
 \rightarrow \Z_2^n=\text{span}\{e_1,\ldots,e_n\}$, which is
 represented by an $n\times m$ matrix $A=(\lambda(F_1),\ldots,\lambda(F_m))$.
  Since $\tilde{v}$ is a fixed point of the $\Z^n_2$-action on $M^n$, we can identify the map
 $\mathfrak{p}^*_{\tilde{v}}$ in~\eqref{Equ:Composite}
  with $p^*: H^1(B\Z^n_2)\rightarrow H^1_{\Z^n_2}(M^n)$ where
 $p: E\Z_2^n \times_{\Z_2^n} M^n\rightarrow B\Z^n_2 $ is the projection.
   Then by the analysis of $p^*$ in \cite[p.438-439]{dj}, we have
 \begin{equation}
   \mathfrak{p}^*_{\tilde{v}}(t_j)=\lambda^*(t_j)=\sum_{\lambda(F_l)=e_j}a_{F_l},
 \end{equation}
 where
 $\lambda^*:\text{span}\{t_1,\ldots,t_n\} \rightarrow \text{span}\{a_{F_1},\ldots,a_{F_m}\}$ is the dual of $\lambda$, which is represented by the transpose $A^t$ of $A$. So we obtain
\begin{equation} \label{Equ:t_j}
   t_j=i_{\tilde{v}}^*(\mathfrak{p}^*_{\tilde{v}}(t_j))=i_{\tilde{v}}^*
   \Big(\sum_{\lambda(F_l)=e_j} a_{F_l}\Big)=\sum_{\lambda(F_l)=e_j} i_{\tilde{v}}^*(a_{F_l}) =
    \sum_{v\in F_l, \, \lambda(F_l)=e_j} i_{\tilde{v}}^*(a_{F_l}).
 \end{equation}
 Observe that among all the $n$ facets of $P^n$ containing $v$, there is only
 one facet (i.e. $F$) colored by $e_j$. So we obtain from~\eqref{Equ:t_j} that
  $$  \sum_{v\in F_l, \, \lambda(F_l)=e_j}
     i_{\tilde{v}}^*(a_{F_l}) = i_{\tilde{v}}^*(a_{F})=t_j.$$
 The lemma is proved.
 \end{proof}

\vskip .1cm

  Now for an arbitrary
   facet $F$ of $P^n$, suppose $\lambda(F) = e_j$. We get from Lemma~\ref{Lem:Local} that
\begin{equation} \label{Equ:i4}
i_4^*(a_F)=\bigoplus_{v\in V(P^n)} i_{\tilde{v}}^*(a_F)= \sum_{v\in F} t_j\cdot\xi_{v} =
  t_{j}\cdot \xi_F.
\end{equation}
Recall that $\xi_v$ denotes the vector in $\mathbb{F}_2^{2r}=\mathbb{F}_2^{|V(P^n)|}$
 with $1$ at the coordinate
corresponding to the vertex $v$ and zero everywhere else.
 Combining~\eqref{Equ:i4} with~\eqref{Equ:Psi} and~\eqref{Equ:t-ti}, we obtain
 \begin{equation}\label{f}
    \psi^* i_4^*(a_F) =t\cdot\xi_F.
  \end{equation}
  So $\psi^* i_4^*(H^1_{\Z_2^n}(M;{\Bbb F}_2))= t\cdot\mathfrak{B}_1(P^n)$ since $\psi^*$ and
   $i_4^*$ are graded ring homomorphisms. Then by~\eqref{Equ:Ident}, we have
   $\mathrm{Im}(i^*_3)=
  \mathrm{Im}(\psi^*\circ i_4^*)=t\cdot\mathfrak{B}_1(P^n)$. This
  implies that $V^M_1=\mathfrak{B}_1(P^n)$.
  So we complete the proof of Theorem~\ref{Thm:Main-1}. \hfill \qed

\vskip .6cm

\section{Binary codes from general simple polytopes}

  The definition of $\mathfrak{B}_k(P^n)$ in~\eqref{Equ:Def-Bk} clearly
  makes sense for an arbitrary $n$-dimensional simple polytope $P^n$.
 We call $\mathfrak{B}_k(P^n) \subset {\Bbb F}_2^{|V(P^n)|}$
  the {\em codimension-$k$ face code of $P^n$}.
 It is obvious that $\mathfrak{B}_0(P^n)=\{\underline{0}, \underline{1}\}\cong {\Bbb F}_2$, and $\mathfrak{B}_n(P^n)\cong {\Bbb F}_2^{|V(P^n)|}$ where
  $$ \underline{0}= (0,\ldots, 0),\ \  \underline{1}=(1,\ldots, 1).$$
 In this section, we study some properties of $\mathfrak{B}_k(P^n)$.
 The arguments in this section
   are completely combinatorial and are independent from the discussion of equivariant 
   cohomology and small covers in the previous sections.

\vskip .1cm

For each $0\leq k\leq n$, $\mathfrak{B}_k(P^n)$ determines a matrix $M_k(P^n)$ with columns
$\xi_f\in \mathfrak{B}_k(P^n)$ with respect to an ordering of all the vertices and
 all the codimension-$k$ faces $f$ of $P^n$.
 We call $M_k(P^n)$ the {\em code matrix of codimension-$k$ faces} of $P^n$.
\vskip .1cm

\begin{exam}
   Under the labeling of the vertices of the $6$-prism $Q^3$ in Figure~\ref{6prism},
the code matrix $M_1(Q^3)$ is a $12\times 8$ binary matrix shown in Figure~\ref{6prism}.
The codewords corresponding to the top facet $F_1$ and the bottom facet $F_2$ are the first and the second columns of $M_1(Q^3)$.
\end{exam}
\begin{figure}[h]
\includegraphics[width=0.5\textwidth]{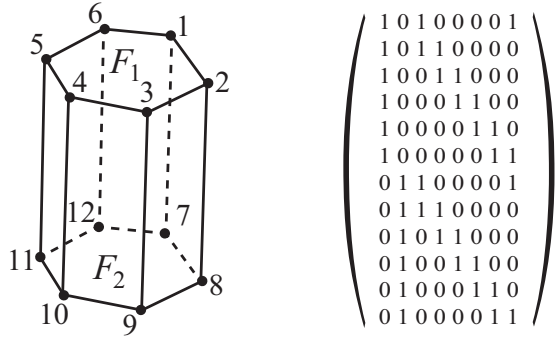}
\caption{A $6$-prism and its code matrix of codimension-one faces}
\label{6prism}
\end{figure}

\vskip .1cm

\begin{prop}
 For any $n$-dimensional simple polytope $P^n$, we have
  $$ \mathfrak{B}_{n-1}(P^n) =\{ u \in \mathbb{F}^{|V(P^n)|}_2 \,|\,
   wt(u) \ \text{is even} \}.$$ 
   So $\dim_{\mathbb{F}_2}\mathfrak{B}_{n-1}(P^n)= |V(P^n)|-1$. 
\end{prop}
  \begin{proof}
      It is clear that for any $1$-face $f$ of $P^n$, its Hamming weight $wt(\xi_f)=2$. So 
      $\mathfrak{B}_{n-1}(P^n)$ is a subspace of $\{u \in \mathbb{F}^{|V(P^n)|}_2 \,|\,
   wt(u) \ \text{is even} \}$. On the other hand, for
    two arbitrary vertices $v,v'\in V(P^n)$, since the $1$-skeleton of $P^n$ is connected,
   there exists a path of $1$-faces $f_1,\cdots, f_k$ on $P^n$ connecting $v$ and $v'$.
   So we have
   $$\xi_{v} + \xi_{v'} = \xi_{f_1}+\cdots +\xi_{f_k} \in \mathfrak{B}_{n-1}(P^n).$$
   Then our lemma follows from the fact that $\{ u \in \mathbb{F}^{|V(P^n)|}_2 \,|\,
   wt(u) \ \text{is even} \}$ is linearly spanned by $\{ \xi_{v} + \xi_{v'} \,|\,
   v,v'\in V(P^n)\}$.
  \end{proof}

 In the following proposition, 
  we obtain a lower bound of the dimension of $\mathfrak{B}_k(P^n)$.

\begin{prop}\label{Prop:dim-bound}
For any $n$-dimensional simple polytope $P^n$, we have
 $$\dim_{\mathbb{F}_2}\mathfrak{B}_k(P^n)\geq h_0(P^n)+\cdots+h_k(P^n), \ 0\leq k\leq n.$$
\end{prop}
\begin{proof}
 Using the Morse-theoretical argument in~\cite{br}, we can define a generic height function $\phi$ on $P^n$ that makes the $1$-skeleton of $P^n$ into a directed graph by
    orienting each edge so that $\phi$ increases along it.
    Then for any face $f$ of $P^n$ with dimension greater than $0$, $\phi|_f$ assumes
    its maximum (or minimun) at a vertex. Since $\phi$ is generic,
    each face $f$ of $P^n$ has a unique ``top'' and a unique ``bottom'' vertex.
   For each vertex $v$ of $P^n$, we define the \emph{index} $\mathrm{ind}(v)$ 
    of $v$ to be the
   number of incident edges of $P^n$ that point towards $v$.
  A simple argument (see~\cite[p.115]{br} or~\cite[p.13]{bp}) shows that for any
  $0 \leq j \leq n$, the number of vertices of $P^n$ with index $j$ is equal to $h_j(P^n)$. \vskip .1cm

 Now fix an integer $0\leq k\leq n$. For
    any vertex $v$ of $P^n$ with $0\leq \mathrm{ind}(v) \leq k$, there are
   exactly $n-\mathrm{ind}(v)$ incident edges of $P^n$ that point away from $v$.
   So there are $\binom{n-\mathrm{ind}(v)}{n-k}$ codimension-$k$ faces of $P^n$ that
   are incident to $v$ and take $v$ as their (unique) ``bottom'' vertex.
    Choose an arbitrary one such face at $v$, denoted by $f^{n-k}_v$.
   
   \vskip .1cm

 \noindent {\bf Claim:} $\{ \xi_{f^{n-k}_v} \,|\, 0\leq \mathrm{ind}(v) \leq k,\, v\in V(P^n) \}$ is a linearly independent subset of $\mathfrak{B}_k(P^n)$.
 \vskip .1cm

 Otherwise there would exist vertices $v_1,\ldots, v_s$ of $P^n$ with
 $0\leq \mathrm{ind}(v_i) \leq k$, $1\leq i \leq s$, so that
  $\xi_{f^{n-k}_{v_1}} + \cdots + \xi_{f^{n-k}_{v_s}} =0$.
  Without loss of generality, we can assume $\phi(v_1) < \cdots < \phi(v_s)$.
 Then among $f^{n-k}_{v_1},\ldots, f^{n-k}_{v_s}$, only $f^{n-k}_{v_1}$ is
  incident to the vertex $v_1$. From this fact, we obtain
  $\xi_{v_1}\circ \big(\xi_{f^{n-k}_{v_1}} + \cdots + \xi_{f^{n-k}_{v_s}}\big) = \xi_{v_1}\circ \xi_{f^{n-k}_{v_1}} = \xi_{v_1} = 0$, which is absurd.
  \vskip .1cm

This claim implies that $\dim_{\mathbb{F}_2} \mathfrak{B}_k(P^n)$ is greater or equal to the number of vertices of $P^n$ whose indices are less or equal to $k$.
Hence $\dim_{\mathbb{F}_2}\mathfrak{B}_k(P^n)\geq h_0(P^n)+\cdots+h_k(P^n)$.
\end{proof}

 \begin{rem}
 Suppose $P^n$ is an $n$-colorable simple
  $n$-polytope. Then the dimension of 
  $\mathfrak{B}_k(P^n)$ is exactly $h_0(P^n)+\cdots+h_k(P^n)$ by corollary~\ref{col-self}.
  So by the claim in the proof of Proposition~\ref{Prop:dim-bound}, 
 $\{ \xi_{f^{n-k}_v} \,|\, 0\leq \mathrm{ind}(v) \leq k,\, v\in V(P^n) \}$
   is actually a linear basis for $\mathfrak{B}_k(P^n)$.
   This gives us an interesting way to write a linear basis of
   $\mathfrak{B}_{k}(P^n)$ from a generic height function on $P^n$.
   In particular when $n$ is odd, we can obtain a linear basis of the self-dual binary code
  $\mathfrak{B}_{\frac{n-1}{2}}(P^n)$ in this way.  
  \end{rem}

  \begin{cor}\label{Cor:dim}
  Let $P^n$ be an $n$-diemensional simple polytope with $m$ facets. Then
     $$\dim_{{\Bbb F}_2} \mathfrak{B}_1(P^n)\geq m-n+1.$$
  Moreover, for any vertex $v$ of $P^n$, let $F_v$ be an arbitrary facet of $P^n$ 
  containing $v$ and $F_1,\ldots, F_{m-n}$ be all the facets of $P^n$ not containing $v$.
  Then $\xi_{F_v},\xi_{F_{1}}, \cdots, \xi_{F_{m-n}}\in \mathfrak{B}_1(P^n)$ 
  are linearly independent.
\end{cor}

\begin{proof} 
  Since $h_0(P^n)=1$, $h_1(P^n)=m-n$, Proposition~\ref{Prop:dim-bound} tells us that 
  $$\dim_{{\Bbb F}_2} \mathfrak{B}_1(P^n)\geq h_0(P^n) + h_1(P^n) = 
  m-n+1.$$
  For any vertex $v$ of $P^n$, we can define a height function $\phi$ as in the
  proof of Proposition~\ref{Prop:dim-bound} so that $v$ is the unique ``bottom'' 
  vertex of $P^n$ relative to $\phi$. Then $v$ is the only vertex of index $0$.
  For each $1\leq i \leq n$, let $v_i$ be the bottom vertex of $F_i$ relative to $\phi$.
  Then it is easy to see that 
   $v_1,\cdots, v_{m-n}$ are exactly all the vertices of index $1$ relative to $\phi$.
   So by the claim in the
  proof of Proposition~\ref{Prop:dim-bound} for $k=1$, $\xi_{F_v}, 
  \xi_{F_{1}}, \cdots, \xi_{F_{m-n}}$ are linearly independent in $\mathfrak{B}_1(P^n)$.
\end{proof}

Next let us look at what happens when $\dim_{{\Bbb F}_2} \mathfrak{B}_1(P^n)= m-n+1$.
\begin{prop}\label{Prop:dim1}
Let $P^n$ be an $n$-dimensional simple polytope with $m$ facets. Then $\dim_{{\Bbb F}_2} \mathfrak{B}_1(P^n)= m-n+1$ if and only if $P^n$ is $n$-colorable.
\end{prop}

\begin{proof}
Let $\{ F_1, \ldots, F_m\}$ be all the facets of $P^n$.
Suppose $P^n$ is $n$-colorable. Then $P^n$ admits a coloring $\lambda: \mathcal{F}(P^n)\rightarrow \Z_2^n$ such that
the image $\mathrm{Im} \lambda$  is a basis  $\{e_1,\ldots,e_n\}$ in $\Z_2^n$.
Set $$\mathcal{F}_i=\{F\in \mathcal{F}(P^n)\,|\, \lambda(F)=e_i \}, i=1,2,\ldots,n.$$
 By the definition of $\lambda$,
  each vertex of $P^n$ is incident to exactly one facet in $\mathcal{F}_i$. So we have
\begin{equation}\label{eq}
\bigcup_{F\in \mathcal{F}_i} V(F)=V(P^n), \ \ \sum_{F\in \mathcal{F}_i}\xi_F
   =\sum_{v\in V(P^n)}\xi_v=\underline{1}.
\end{equation}

Without loss of generality, assume that the facets $F_1, \ldots, F_n$ meet at a vertex, and $\lambda(F_i)=e_i$, $1\leq i \leq n$. 
 So by our definition, $F_i\in \mathcal{F}_i$, $1\leq i \leq n$.
 We claim that
 for each $1\leq i\leq n-1$, $\xi_{F_i}$ can be written as a linear combination of elements in 
 $\xi_{F_n}, \xi_{F_{n+1}}, \ldots, \xi_{F_m}$.
  Indeed, it follows from~\eqref{eq} that
 \begin{equation} \label{Equ:F-eq}
   \sum\limits_{F\in \mathcal{F}_i} \xi_{F} + \sum\limits_{F\in \mathcal{F}_n} \xi_{F}=\underline{1}+\underline{1}=\underline{0}.
   \end{equation}
Observe that
$$ \{\xi_F| F\in\mathcal{F}_i\}\subset \{\xi_{F_i},\xi_{F_n},\xi_{F_{n+1}},\ldots
\xi_{F_m}\},\ \ \{\xi_F| F\in\mathcal{F}_n\}\subset \{\xi_{F_i},\xi_{F_n},\xi_{F_{n+1}},\ldots
\xi_{F_m}\}.$$
 So~\eqref{Equ:F-eq} implies that $\xi_{F_i}$ is equal to
  a linear combination of elements in $\xi_{F_n}, \xi_{F_{n+1}}, \ldots, \xi_{F_m}$.
 Moreover, we know from Corollary~\ref{Cor:dim} that
$\xi_{F_n}, \xi_{F_{n+1}}, \ldots, \xi_{F_m}$ are linearly independent, hence form a basis of $\mathfrak{B}_1(P^n)$.
 So $\dim_{{\Bbb F}_2} \mathfrak{B}_1(P^n)= m-n+1$.

\vskip .2cm
Conversely, suppose $\dim_{{\Bbb F}_2} \mathfrak{B}_1(P^n)=m-n+1$.
If $P^n$ is not $n$-colorable, by Theorem~\ref{j} there exists a $2$-face $f^2$ of $P^n$ which
has odd number of vertices, say $v_1,\ldots,v_{2k+1}$. Without loss of generality, assume that  $f^2=F_1\cap\cdots\cap F_{n-2}$ and $v_1=F_1\cap F_2\cap\cdots\cap F_n$. By 
Corollary~\ref{Cor:dim}, $\{ \xi_{F_n},\xi_{F_{n+1}}, \ldots , \xi_{F_m} \}$ is
  a basis of $\mathfrak{B}_1(P^n)$.
Without loss of generality, we may assume the following (see Figure~\ref{p:Proof-1})
\begin{equation*}\label{2}
\begin{cases}
v_1=f^2\cap F_{n-1}\cap F_n\\
  v_2=f^2\cap F_n\cap F_{n+1}\\
 \quad \cdots\\
  v_i=f^2\cap F_{n+i-2}\cap F_{n+i-1}\\
  \quad \cdots\\
  v_{2k}=f^2\cap F_{n+2k-2} \cap F_{n+2k-1}\\
  v_{2k+1}=f^2\cap F_{n+2k-1} \cap F_{n-1}\\
\end{cases}
\end{equation*}

\begin{figure}[h]
\includegraphics[width=0.35\textwidth]{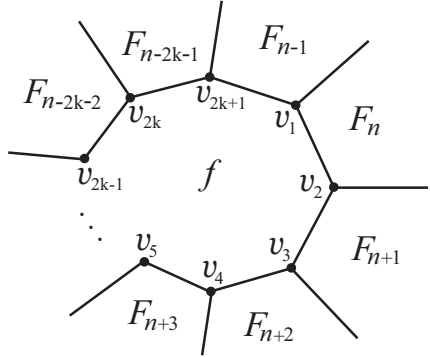}
\caption{A face $f$ with odd number of vertices}
\label{p:Proof-1}
\end{figure}

Assume that there exists $\epsilon_i\in {\Bbb F}_2$, $i=1, n, \ldots , m$ so that
\begin{equation}\label{2}
  \epsilon_1\xi_{F_1}+\epsilon_n\xi_{F_n}+\epsilon_{n+1}\xi_{F_{n+1}}+\cdots++\epsilon_m\xi_{F_m}
   =\underline{0}.
\end{equation}
 For each $i$, by taking the inner product with $\xi_{v_i}$
  on both sides of (\ref{2}),  we get
 \begin{equation}\label{3}
   \begin{cases}
  \epsilon_1 +\epsilon_n =0,\\
  \epsilon_1+\epsilon_n+\epsilon_{n+1} =0,\\
  \epsilon_1+\epsilon_{n+1}+\epsilon_{n+2} =0,\\
  \quad \cdots\\
  \epsilon_1+\epsilon_{n+i-2}+\epsilon_{n+i-1} =0,\\
  \quad \cdots\\
  \epsilon_1+\epsilon_{n+2k-2}+\epsilon_{n+2k-1} =0,\\
  \epsilon_1+\epsilon_{n+2k-1} =0.\\
   \end{cases}
    \end{equation}
The coefficient matrix of the above linear system is a $(2k+1)\times (2k+1)$ matrix
over ${\mathbb F}_2$.
 $$ \qquad\qquad
\begin{pmatrix}
1&1&0&0&0&\cdots&0&0&0\\
1&1&1&0&0&\cdots&0&0&0\\
1&0&1&1&0&\cdots&0&0&0\\
1&0&0&1&1&\cdots&0&0&0\\
&&&&\cdots\cdots\\
1&0&0&0&0&\cdots&1&1&0\\
1&0&0&0&0&\cdots&0&1&1\\
1&0&0&0&0&\cdots&0&0&1
\end{pmatrix}_{(2k+1)\times (2k+1)}$$
 It is easy to show that the determinant of this matrix is $1$. So
 the linear system~\eqref{3} only has zero solution, which implies that
  $\xi_{F_1},\xi_{F_n}, \ldots, \xi_{F_m}$ are linearly independent.
 Then we have $\dim_{{\Bbb F}_2} \mathfrak{B}_1(P^n)\geq m-n+2$. But this contradicts our
  assumption that $\dim_{{\Bbb F}_2} \mathfrak{B}_1(P^n)=m-n+1$. So the proposition is proved.
\end{proof}

 From the above discussion, we can derive several new criteria to judge whether
 a simple $n$-polytope $P^n$ is $n$-colorable in terms of
 the associated binary codes $\{ \mathfrak{B}_k(P^n)\}_{0\leq k \leq n}$.\vskip .1cm

\begin{prop}\label{collection}
Let $P^n$ be an $n$-dimensional simple polytope with $m$ facets.
Then the following statements are equivalent.
\begin{itemize}
 \item[(1)] $P^n$ is $n$-colorable. \vskip .1cm
 \item[(2)] There exists  a partition $\mathcal{F}_1, \ldots, \mathcal{F}_n$ of the set $\mathcal{F}(P^n)$ of all facets, such that for each $1\leq i\leq n$, all the
 facets in $\mathcal{F}_i$ are pairwise disjoint and $\sum_{F\in \mathcal{F}_i}\xi_F=\underline{1}$ (i.e.,
   each vertex of $P^n$ is incident to exactly one facet from every $\mathcal{F}_i$). \vskip .1cm
 \item[(3)] $\mathfrak{B}_0(P^n) \subset \mathfrak{B}_1(P^n) \subset\cdots\subset \mathfrak{B}_{n-1}(P^n)\subset \mathfrak{B}_n(P^n)\cong {\Bbb F}_2^{|V(P^n)|}.$ \vskip .1cm
 \item[(4)] $\mathfrak{B}_{n-2}(P^n)\subset \mathfrak{B}_{n-1}(P^n)$. \vskip .1cm
 \item[(5)] $\dim_{{\mathbb F}_2} \mathfrak{B}_1(P^n)= m-n+1$.
  \end{itemize}
\end{prop}

\begin{proof}
 It is easy to verify the above equivalences when $n\leq 2$. So we assume $n\geq 3$ below.
 In the proof of Proposition~\ref{Prop:dim1}, we have proved
  $(1)\Rightarrow (2)$ and $(1)\Leftrightarrow (5)$.

\vskip .2cm
Now we show that $(2)\Rightarrow (3)$. By the condition in (2), we clearly have
 $$\mathfrak{B}_0(P^n) \subset \mathfrak{B}_1(P^n), \ \
  \mathfrak{B}_{n-1}(P^n) \subset \mathfrak{B}_{n}(P^n).$$
It remains to show that $\mathfrak{B}_k(P^n) \subset \mathfrak{B}_{k+1}(P^n)$ for
each $1\leq k\leq n-2$. Let $f^{n-k}$ be a codimension-$k$ face of $P^n$.
Without the loss of generality, we assume that
   $$f^{n-k}=F_1\cap F_2\cap\cdots\cap F_k, \ \text{where}\
     F_i\in \mathcal{F}_i,\ i=1, \ldots , k.$$
For each $j=k+1,\ldots,n$, we have that
 $$\sum_{F\in \mathcal{F}_j}\xi_{F\cap f^{n-k}}=
   \sum_{F\in \mathcal{F}_j}\xi_{F}\circ\xi_{f^{n-k}}
  =\xi_{f^{n-k}}\circ(\sum_{F\in \mathcal{F}_j}\xi_{F} )
  =\xi_{f^{n-k}}\circ \underline{1}=\xi_{f^{n-k}}.$$
In the above equality, if $F\cap f^{n-k}=\emptyset$, then $\xi_{F\cap f^{n-k}}=\xi_\emptyset=\underline{0}$.
If $F\cap f^{n-k}\neq \emptyset$, then $F\cap f^{n-k}$ is a face of
  codimension $k+1$.
so $\xi_{F\cap f^{n-k}}\in \mathfrak{B}_{k+1}(P^n)$.
Thus we get $\xi_{f^{n-k}}=\sum_{F\in \mathcal{F}_j}\xi_{F\cap f^{n-k}} \in \mathfrak{B}_{k+1}(P^n)$.
This completes the proof of $(2)\Rightarrow (3)$.

\vskip .2cm

 It is trivial that $(3)\Rightarrow (4)$. Next we show $(4)\Rightarrow (1)$.
Assume $\mathfrak{B}_{n-2}(P^n)\subset \mathfrak{B}_{n-1}(P^n)$.
Notice that the number of nonzero coordinates in any vector in
 $\mathfrak{B}_{n-1}(P^n)$ must be even.
 So for any $2$-face $f^2$ of $P^n$, we have
  $\xi_{f^{2}}\in \mathfrak{B}_{n-2}(P^n)\subset \mathfrak{B}_{n-1}(P^n)$, which
  implies that $f^2$ has an even number vertices. Hence $P^n$ is $n$-colorable 
  by Theorem~\ref{j}.
\end{proof}

\vskip .6cm

\section{Self-dual binary codes from general
 simple polytopes}

In this section we discuss under what conditions can $\mathfrak{B}_k(P^n), 0\leq k\leq n$, be a self-dual code in $\mathbb{F}^{|V(P^n)|}_2$. It is clear that when the number of vertices 
$|V(P^n)|$ of $P^n$ is odd, $\mathfrak{B}_k(P^n)$ cannot be a self-dual code for any $k$. 

\begin{lem}\label{ineq}
Let $P^n$ be an $n$-dimensional simple polytope with $n\geq 3$.
Assume that $\mathfrak{B}_k(P^n)$ is a self-dual code. Then 
$\underline{1}\in \mathfrak{B}_k(P^n)$ and $0<2k<n$. 
\end{lem}
\begin{proof}
 Since $\mathfrak{B}_k(P^n)$ is self-dual, we have
  $\dim_{\mathbb{F}_2} \mathfrak{B}_k(P^n)={{h_1(P^n)+\cdots+h_n(P^n)}\over 2}$. Moreover,  from \cite[Corollary 3.1]{cl} it is easy to see that $\underline{1}\in \mathfrak{B}_k(P^n)$. Obviously, $k=0$ is impossible since $\dim \mathfrak{B}_0(P^n)=1$ and $n\geq 3$.
 If $2k\geq n$, then at any vertex $v$ of $P^n$ there exist two codimension-$k$ faces $f_1$ and $f_2$ of $P^n$ such that $f_1\cap f_2=v$. But then $\langle \xi_{f_1},\xi_{f_2}\rangle=1$ which contradicts the assumption that $\mathfrak{B}_k(P^n)$ is self-dual. We can also prove 
 $2k<n$ using Proposition~\ref{Prop:dim-bound}. Indeed, by Proposition~\ref{Prop:dim-bound},
  we have
 \begin{equation} \label{Equ:half}
   \dim_{\mathbb{F}_2} \mathfrak{B}_k(P^n)={{h_1(P^n)+\cdots+h_n(P^n)}\over 2} \geq
  h_0(P^n)+\cdots+h_k(P^n).
 \end{equation} 
  Then since $h_i(P^n)>0$ and $h_i(P^n)=h_{n-i}(P^n)$ (Dehn-Sommerville relations)
   for all $0\leq i \leq n$, we must have $2k< n$. 
   \end{proof} 

   \begin{thm} \label{thm:Self-Dual}
  For an $n$-dimensional simple polytope $P^n$ with $n\geq 3$, 
 $\mathfrak{B}_k(P^n)$ is a self-dual code if and only if the following two conditions are satisfied:
 \begin{itemize}
  \item[(a)] $|V(P^n)|$ is even and 
      $\dim_{\mathbb{F}_2}\mathfrak{B}_k(P^n) = \frac{|V(P^n)|}{2}$.\vskip .1cm
  \item[(b)] All faces of codimensions $k, \ldots , 2k$ in $P^n$ have an even number of vertices.
\end{itemize}
\end{thm}

\begin{proof}
  If $\mathfrak{B}_k(P^n)$ is a self-dual code, then (a) obviously holds. Let $|V(P^n)|=2r$.
 For any face $f$ of codimension $l$ where $k\leq l\leq 2k$, we can always write 
 $f=f_1\cap f_2$ where $f_1$ and $f_2$ are faces of codimension $k$. In particular if $f$ is of codimension $k$, we just let $f_1=f_2=f$. Then 
  $\xi_f=\xi_{f_1}\circ \xi_{f_2} \in \mathcal{V}_{2r}$ since $\mathfrak{B}_k(P^n)$ is self-dual
  (see Lemma~\ref{code}). This implies that the number of vertices of $f$ is even.
  \vskip .1cm

   Conversely, suppose $\mathfrak{B}_k(P^n)$ satisfies (a) and (b). For any
  codimension $k$ faces $f$ and $f'$ of $P^n$, either $f\cap f'=\varnothing$ or 
  the codimension of $f\cap f'$ is between $k$ and $2k$. Then by (b), the number of vertices of
   $f\cap f'$ is even, which
   implies that $\langle \xi_f, \xi_{f'} \rangle = 0$. Then by Lemma~\ref{code},
  $\mathfrak{B}_k(P^n)$ is self-dual in $\mathbb{F}^{|V(P^n)|}_2$. Note that by~\eqref{Equ:half},
   the condition (a) implies $0<2k<n$ when $n\geq 3$. 
\end{proof} 

    \begin{rem}
   When $k \geq \frac{n-2}{2}$ ($n\geq 3$), the condition (b) in Theorem~\ref{thm:Self-Dual} 
   implies that the polytope $P^n$ is $n$-colorable. But 
    if $k<\frac{n-2}{2}$, the condition (b) cannot guarantee that $P^n$ is $n$-colorable.
    For example let $P^n =\Delta^2\times [0,1]^{n-2}$, $n\geq 3$, 
     where $\Delta^2$ is the $2$-simplex. Then $P^n$ satisfies 
   the condition (b) for all $k<\frac{n-2}{2}$ because
     any face of $P^n$ with dimension greater than $2$ has an even number of vertices.
   But by Theorem~\ref{j}(b),
    $P^n$ is not $n$-colorable since $\Delta^2$ is a $2$-face of $P^n$.
    \end{rem}

 \noindent \textbf{Problem 2:} 
   For an arbitrary simple polytope $P^n$,  
   determine the dimension of $\mathfrak{B}_k(P^n)$ for all
   $0\leq k \leq n$. 
   \vskip .1cm
   
   We have seen in Corollary~\ref{col-self} 
   that when a simple $n$-polytope $P^n$ is $n$-colorable,
   the dimension of $\mathfrak{B}_k(P^n)$ can be expressed by the $h$-vector of $P^n$.
   But generally, the answer to this problem is not clear to us.\vskip .1cm

 \begin{prop}\label{self-poly}
 Let $P^n$ be a simple $n$-polytope  with $2r$ vertices and $m$ facets and $n\geq 3$.
\begin{itemize}
\item[(a)] When $n=3$, $\mathfrak{B}_k(P^3)$ is a self-dual code if and only if $k=1$ and $P^3$ is 3-colorable.
  \item[(b)] When $n=4$, $\mathfrak{B}_k(P^4)$ is never self-dual for any $0\leq k \leq n$.
\item[(c)] When $n=5$,  $\mathfrak{B}_k(P^5)$ is a self-dual code if and only if $k=2$ and $P^5$ is 5-colorable.
\item[(d)] When $n>5$, if $\mathfrak{B}_k(P^n)$ is a self-dual code and $m> {{(n+1)(n-2)}\over {n-3}}$, then $k\geq 2$.
\end{itemize}
\end{prop}

\begin{proof}
 When $n=3$, by Lemma~\ref{col-self} it suffices to show that if  $\mathfrak{B}_k(P^3)$ is a self-dual code, then $k=1$ and $P^3$ is $3$-colorable. Assume that $\mathfrak{B}_k(P^3)$ is a self-dual code.
   Then by Lemma~\ref{ineq},  $k$ must be $1$ and any
   $2$-face of $P^3$ has an even number of vertices. So $P^3$ is 3-colorable by Theorem~\ref{j}.
        This proves (a).
 \vskip .1cm
 When $n=4$, assume that  $\mathfrak{B}_k(P^4)$ is a self-dual code. By Lemma~\ref{ineq},  $k$ must be $1$ and  any $2$-face has an even number of vertices. So $P^4$ is $4$-colorable
 by Theorem~\ref{j}. Then (b) follows from Lemma~\ref{col-self}.
\vskip .1cm

  Now let $n\geq 5$ and assume that $\mathfrak{B}_k(P^n)$ is a self-dual code.
  Let $f_{k-1}(P^n)$ denote the number of codimension-$k$ faces in $P^n$.
  Then by~\cite[Theorems 1.33 and Theorem 1.37]{bp}, we have
  $$f_{k-1}(P^n)\leq {{m}\choose k} \ \text{if}\ 2k<n, \ \text{and}\ f_{n-1}(P^n)=2r\geq (n-1)m-(n+1)(n-2).$$
  By the fact that $\dim_{\mathbb{F}_2} \mathfrak{B}_k(P^n) \leq f_{k-1}(P^n)$, we obtain
\begin{equation} \label{Equ:Dim-Inequ}
  (n-1)m-(n+1)(n-2)\leq 2r=2 \dim_{\mathbb{F}_2} \mathfrak{B}_k(P^n)
 =h_0(P^n)+\cdots+h_n(P^n) \leq 2{{m}\choose k}.
 \end{equation}
If $k=1$, we have $m\leq {{(n+1)(n-2)}\over {n-3}}$. Thus, if $m> {{(n+1)(n-2)}\over {n-3}}$, then $k\geq 2$. This proves (d).

\vskip .1cm

Next we consider the case $n=5$ with $k=1$.
 In this case, we have $6\leq m\leq 9$ and by Lemma~\ref{ineq},
 all $3$-faces and $4$-faces of $P^n$ have even of vertices.
 Moreover, $\dim_{\mathbb{F}_2} \mathfrak{B}_k(P^n) \leq f_{k-1}(P^n)$ implies that
 $$r=h_0(P^5) +h_1(P^5) +h_2(P^5) =1+m-5+h_2(P^5)=m-4+h_2(P^5)\leq m.$$
 So we have $h_2(P^5)\leq 4$. In addition, by g-theorem we have that
 $h_2(P^5)\geq h_1(P^5)=m-5$, and $h_1(P^5)-h_0(P^5)\geq h_2(P^5)-h_1(P^5)$ so
  $h_2(P^5)\leq 2h_1(P^5)-h_0(P^5)=2m-11$.
 Combining all these restrictions, we
  can list all such simple $5$-polytopes in terms of their $h$-vectors as follows:
\begin{itemize}
 \item[$\bullet$] $h(P_1^5)=(1, 4, 4, 4, 4, 1)$ with $m=9$;
 \item[$\bullet$] $h(P_2^5)=(1, 3, 4, 4, 3, 1)$ with $m=8$;
\item[$\bullet$]  $h(P_3^5)=(1, 3, 3, 3, 3, 1)$ with $m=8$;
\item[$\bullet$] $h(P_4^5)=(1, 2, 3, 3, 2, 1)$ with $m=7$;
\item[$\bullet$] $h(P_5^5)=(1, 2, 2, 2, 2, 1)$ with $m=7$;
\item[$\bullet$] $h(P_6^5)=(1,1,1,1,1,1)$ with $m=6$.
\end{itemize}

  Clearly, $P^5_6$ is just a $5$-simplex. So $P_6^5$ should be excluded since each facet of
$P_6^5$ is a $4$-simplex which has an odd number of vertices.
By~\cite[Theorem 1.37]{bp}, a direct check shows that
$P_1^5, P_3^5, P_5^5$ are all stacked 5-polytopes. Then they can also be excluded
 since any stacked 5-polytope has at least one $4$-simplex as its facet. Recall that a simple $n$-polytope
$S$ is called {\em stacked} if there is a sequence $S_0,  S_1,\ldots,   S_l = S$ of simple
$n$-polytopes such that $S_0$ is an $n$-simplex and $S_{i+1}$ is
obtained from $S_i$ by cutting a vertex of $S_i$ (cf. \cite[Definition 1.36]{bp}).

\vskip .1cm
By \cite[Theorem 1.33]{bp}, we can directly check that $P_4^5$ is the dual polytope of a cyclic polytope $C^5(7)$. Let $\{F_1, ..., F_{7}\}$ be the set of all facets of $P_4^5$.
By the main theorem in \cite{sh}, we can write all the $12$
vertices $v_1, ..., v_{12}$ of $P_4^5$
explicitly in terms of the intersections of its facets $F_1, ..., F_7$ as follows:
\begin{align*}
& v_1=F_1\cap F_2\cap F_3\cap F_4\cap F_5, \ \ \ v_2=F_1\cap F_2\cap F_3\cap F_4\cap F_7, \\
& v_3=F_1\cap F_2\cap F_3\cap F_6\cap F_7, \ \ \ v_4=F_1\cap F_2\cap F_5\cap F_6\cap F_7,\\
& v_5=F_1\cap F_4\cap F_5\cap F_6\cap F_7, \ \ \ v_6=F_3\cap F_4\cap F_5\cap F_6\cap F_7,\\
& v_7=F_1\cap F_3\cap F_4\cap F_5\cap F_6, \ \ \ v_8=F_2\cap F_3\cap F_4\cap F_5\cap F_7,\\
& v_9=F_1\cap F_2\cap F_4\cap F_5\cap F_7, \ \ \ v_{10}=F_1\cap F_3\cap F_4\cap F_6\cap F_7,\\
& v_{11}=F_1\cap F_2\cap F_3\cap F_5\cap F_6, \ \  v_{12}=F_2\cap F_3\cap F_5\cap F_6\cap F_7.
\end{align*}
 We can easily see that each of $F_1, F_3, F_5, F_7$ has $9$ vertices,
 and each of $F_2, F_4, F_6$ has $8$ vertices. Thus, $P_4^5$ should be excluded as well.

 \vskip .1cm
 Now the only case left to check 
 is $P_2^5$. Note that the dual polytope of $P_2^5$ is a simplicial
  $5$-polytope with $8$ vertices and $16$ facets. 
 By the classification in~\cite[\S 6.3, p.108-112]{gr}, there are exactly $8$ 
 simplicial $5$-polytopes with $8$ vertices up to combinatorial equivalence. 
 They are listed in \cite[\S 6.3, p.112]{gr} in terms of standard contracted Gale-diagrams. 
 By examining those Gale-diagrams,
   we find that only two of them (shown in Figure~\ref{p:Gale-Diagram})
  give simplicial $5$-polytopes with $8$ vertices and $16$ facets. Let $Q_1$ 
 be the simplicial $5$-polytope corresponding to the left, and $Q_2$ to the right
 diagram in Figure~\ref{p:Gale-Diagram}.
   A simple calculation shows that
   $\dim_{\mathbb{F}_2} \mathfrak{B}_1(Q_1^*)=6\neq 8$ and $Q_2^*$ has a facet with $11$ 
   vertices. Hence $P_2^5$ cannot be $Q_1^*$ or $Q_2^*$. 
   So $P_2^5$ should be excluded as well.
   
   \begin{figure}[h]
\includegraphics[width=0.43\textwidth]{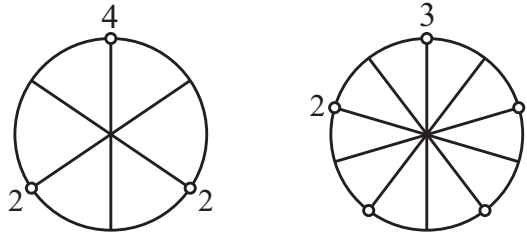}
\caption{Two standard contracted Gale-diagrams}
\label{p:Gale-Diagram}
\end{figure}

 \vskip .1cm
 Combining the above arguments, we can conclude that when $n=5$, $\mathfrak{B}_1(P^5)$ is never self-dual. Therefore, if $\mathfrak{B}_k(P^5)$ is self-dual, $k$ must be $2$ and
 so $P^5$ is $5$-colorable by Lemma~\ref{ineq}. Then (c) follows from Lemma~\ref{col-self}.
 \end{proof}

 \begin{cor}
  For an $n$-dimensional simple polytope $P^n$ with $2r$ vertices and $m$-facets, 
   if $\mathfrak{B}_k(P^n)$ is 
   a self-dual code in $\mathbb{F}_2^{2r}$ and $r \geq {m\choose l}$ for some $l< \frac{m-1}{2}$, 
   then $k\geq l$.
 \end{cor}
 \begin{proof}
  By~\eqref{Equ:Dim-Inequ}, we have $\binom{m}{k}\geq r \geq {m\choose l}$. Then since
   $2k< n\leq m-1$ (by Lemma~\ref{ineq}), we obtain $k\geq l$.
 \end{proof}\vskip .1cm

   In general, judging the existence of self-dual codes $\mathfrak{B}_k(P^n)$ for
a non-$n$-colorable simple $n$-polytope $P^n$ seems to be a quite hard
 problem when $n>5$. On the other hand, Corollary~\ref{col-self} tells us that 
 $2k$-colorable simple
 $2k$-polytopes cannot produce any self-dual codes.
 Then considering the statements in Proposition~\ref{self-poly},
  it is reasonable to pose the following conjecture.\vskip .2cm

\noindent {\bf Conjecture.} {\em Let $P^n$ be a simple $n$-polytope  with $2r$ vertices and $m$ facets, where $n\geq 3$. Then
$\mathfrak{B}_k(P^n)$ is a self-dual code if and only if $P^n$ is $n$-colorable,  $n$ is odd and $k={{n-1}\over 2}$.}

 \vskip .6cm

 \section{Minimum distance of self-dual codes from $3$-dimensional simple polytopes}

 \begin{prop} \label{prop:3-polytope}
 For any $3$-dimensional $3$-colorable simple polytope $P^3$,
 the minimum distance of the self-dual code $\mathfrak{B}_1(P^3)$ is always equal to $4$.
 \end{prop}
\begin{proof}
 It is well known that any $3$-dimensional simple polytope must have a $2$-face
with less than $6$ vertices. Then since $P^3$ is even, there must be a $4$-gon $2$-face
 in $P^3$. So by Corollary~\ref{Cor:Main-2}, the minimum distance of $\mathfrak{B}_1(P^3)$ is
  less or equal to $4$. In addition,
  we know that the Hamming weight of any element in $\mathfrak{B}_1(P^3)$
  is an even integer. So we only
  need to prove that for any $2$-face $F_1,\ldots, F_k$ of $P^3$,
  the Hamming weight of $\alpha=\xi_{F_1}+\cdots + \xi_{F_k}\in \mathfrak{B}_1(P^3)$ cannot be $2$.
  \vskip .1cm
 $\bullet$ Let $V(\alpha)$ be the union of all the vertices of $F_1,\ldots, F_k$.
 \vskip .1cm

 $\bullet$ Let $\Gamma(\alpha)$ be the union of all the vertices and edges of
   $F_1,\ldots, F_k$.
  So $\Gamma(\alpha)$ is a graph with vertex set $V(\alpha)$.

  \begin{figure}[h]
\includegraphics[width=0.61\textwidth]{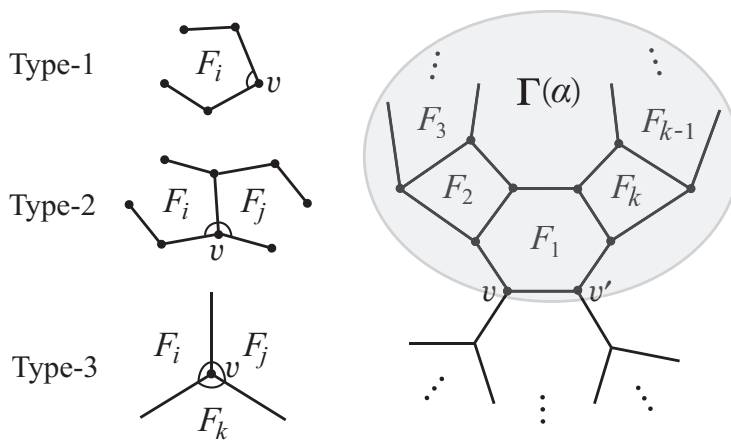}
\caption{The graph of a simple $3$-polytope}
\label{p:Graph}
\end{figure}

  \vskip .1cm

  A vertex $v$ in $V(\alpha)$ is called \emph{type-$j$} if $v$ is incident to
  exactly $j$ facets in $F_1,\ldots, F_k$. Then since $P^3$ is simple,
 any vertex in $V(\alpha)$ is of type-$1$, type-$2$ or type-$3$ (see Figure~\ref{p:Graph}).
 Suppose there are $l_j$ vertices of type-$j$ in $V(\alpha)$, $j=1,2,3$.
 It is easy to see that
  the Hamming weight of $\alpha$ is equal to $l_1+l_3$. Assume that $wt(\alpha)=l_1+l_3 = 2$.
 Then we have three cases for $l_1$ and $l_3$:
   $$
    \mathrm{(a)}\ l_1=2, l_3=0;\ \ \ \mathrm{(b)} \ l_1=1, l_3=1; \ \ \
     \mathrm{(a)}\ l_1=0, l_3=2. $$

  Note that any vertex of type-$2$ or type-$3$ in $V(\alpha)$ meets exactly three edges
  in $\Gamma(\alpha)$. In other words,
   $\Gamma(\alpha)$ is a graph whose vertices are all $3$-valent except the type-$1$ vertices.
  Let $\Gamma(P^3)$ denote the graph of $P^3$ (the union of all the vertices  and edges of $P^3$). and let $\overline{\Gamma}(\alpha) = \Gamma(P^3)\backslash \Gamma(\alpha)$.
Observe that
  $\Gamma(\alpha)$ meets $\overline{\Gamma}(\alpha)$ only at the type-$1$ vertices in $V(\alpha)$.
  \vskip .1cm

 \begin{itemize}
    \item In the case (a), there are two type-$1$ vertices in
   $V(\alpha)$, denoted by $v$ and $v'$.
   Then since $\Gamma(\alpha)$ meets $\overline{\Gamma}(\alpha)$ only at $\{v, v'\}$, removing
    $v$ and $v'$ from the graph $\Gamma(P^3)$ will disconnect $\Gamma(P^3)$
    (see Figure~\ref{p:Graph} for example).
    But according to Balinski's theorem (see~\cite{bal}),
  the graph of any $3$-dimensional simple polytope is a $3$-connected
  graph (i.e. removing any two vertices from the graph does not disconnect it).
  So (a) is impossible.\vskip .1cm

  \item In the case (b), there is only one type-$1$ vertex in $V(\alpha)$, denoted by $v$.
  By the similar argument as above, removing
    $v$ from the graph $\Gamma(P^3)$ will disconnect $\Gamma(P^3)$. This
    contradicts the $3$-connectivity of $\Gamma(P^3)$. So (b) is impossible either.\vskip .1cm

     \item In the case (c), there are no type-$1$ vertices in $V(\alpha)$. So
  $\Gamma(\alpha)$ is a $3$-valent graph. This implies that $\Gamma(\alpha)$
  is the whole $1$-skeleton of $P^3$, and so $V(\alpha)=V(P^3)$. Then the Hamming wight
  $wt(\alpha) = wt(\xi_{F_1}+\cdots + \xi_{F_k}) = wt(\underline{1}) = |V(P^3)| \geq 4$.
   But this
  contradicts our assumption that $wt(\alpha)=2$. So (c) is impossible.
  \end{itemize}

  Therefore, the Hamming weight of any element of $\mathfrak{B}_1(P^3)$
  cannot be $2$. So we finish the proof of the theorem.
  \end{proof}

  \begin{rem}
    It is shown in~\cite{Iz} that any $3$-dimensional $3$-colorable simple polytope
    can be obtained from the $3$-dimensional cube
    via two kinds of operations. So it might be possible
    to classify all the self-dual binary codes obtained from $3$-dimensional simple polytopes. 
    But we would expect the classification to be very complicated.
  \end{rem}
   \vskip .6cm

\section{Properties of $n$-dimensional $n$-colorable simple polytopes}

 For brevity, we use the words ``\emph{even polytope}'' to refer to an
  $n$-dimensional $n$-colorable simple polytope in the rest of the paper.
  Indeed, this term has already been used by Joswig~\cite{jos}.

\begin{defn} (\cite{perles} and \cite[Remark 2]{kalai})
Let $F$ be a facet of a simple polytope $P$ and $V(F)$ be the set of vertices of $F$.
 Define a map $\Xi_F: V(F)\rightarrow V(P)\setminus V(F)$ as follows.
 For each $v\in V(F)$, there is exactly one edge $e$ of $P$, such that $e\nsubseteq F$, $v\in e$
 (since $P$ is simple and $F$ is codimension one). Then let $\Xi_F(v)$ be
   the other endpoint of $e$.
\end{defn}
 \vskip .1cm

\begin{exam}
 Let $Q$ be the $6$-prism in Figure~\ref{6prism} and $F$ be the facet with vertex set
  $\{3,4,9,10\}$. Then by definition, $\Xi_F: \{3,4,9,10 \}\rightarrow \{1,2,5,6,7,8,11,12\}$ where
$$ \Xi(3)= 2,\ \Xi(4)= 5,\ \Xi(9)= 8,\ \Xi(10)= 11. $$
\end{exam}

 \vskip .1cm

\begin{figure}[h]
\includegraphics[width=0.37 \textwidth]{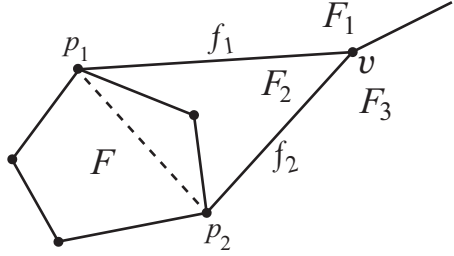}
\caption{A facet $F$ with $\Xi_F$ non-injective}
\label{p:proof-2}
\end{figure}
\vskip .1cm

\begin{prop} \label{Prop:inject}
 For an even polytope $P$, the map
 $\Xi_F$ is injective for any facet $F$ of $P$.
\end{prop}
\begin{proof}
Assume $\Xi_F$ is not injective.
There must exist two vertices $p_1,p_2\in F$ and a vertex
$v\notin F$ such that $v$ is connected to both
$p_1$ and $p_2$ by edges in $P$ (see Figure~\ref{p:proof-2}).
 Let $f_i$ be the edge
 with end points $p_i$ and $v$, $i=1,2$.
  Suppose the dimension of $P$ is $n$. Then there exist
  $n$ facets, $F_1,F_2,\ldots,F_n$, distinct to $F$, such that
$$v=\bigcap_{i=1}^nF_i,\ \  f_1=\bigcap_{i=1}^{n-1}F_i,\ \ f_2=\bigcap_{i=2}^nF_i.$$

Then we have
 $$p_1=F\bigcap \big(\bigcap_{i=1}^{n-1}F_i \big), \ \
    p_2=F\bigcap \big(\bigcap_{i=2}^{n}F_i\big). $$

Since $P$ is $n$-colorable, we can color all the facets of $P$ by
 $n$-colors $e_1,\ldots, e_n$ such that no adjacent facets are assigned
  the same color. Suppose $F_i$ is colored by $e_i$, $i=1,\ldots, n$. Then at $p_1$,
  $F$ has to be colored by $e_n$ while at $p_2$, $F$ has to be colored by $e_1$, contradiction.
\end{proof}

\vskip .1cm

\begin{prop}\label{lowerbound}  Let $P$ be an even polytope.
 For any facet $F$ of $P$, we have
 $$ |V(P)|\geq 2 |V(F)|.$$
 Moreover, $|V(P)|= 2|V(F)|$ if and only if
$P=F\times [0,1]$ where $[0,1]$ denotes a $1$-simplex.
\end{prop}

\begin{proof}
 By Proposition~\ref{Prop:inject}, the map
    $\Xi_F: V(F)\rightarrow V(P)\setminus V(F)$ is injective. So we have
$$|V(F)|\leq |V(P)\setminus V(F)|=|V(P)|-|V(F)|.$$
So $|V(P)|\geq 2|V(F)|$. If $|V(P)|= 2|V(F)|$, the injectivity of $\Xi_F$ implies
 $P=F\times [0,1]$.
\end{proof}

\vskip .1cm

\begin{cor} \label{Cor:bound}
Let $f$ be a codimension-$k$ face of an even polytope $P$. Then $|V(P)| \geq 2^k|V(f)|$.
Moreover,
$|V(P)| = 2^k|V(f)|$ if and only if $P=f\times [0,1]^k$.
\end{cor}

\vskip .1cm

\begin{cor}\label{min}
 For any $n$-dimensional even polytope $P$, we must have
$|V(P)|\geq 2^n$.
 In particular, $|V(P)|=2^n$ if and only if $P=[0,1]^n$ (the $n$-dimensional cube).
\end{cor}

\vskip .1cm

\begin{cor}\label{Cor:3-face}
 Suppose $P$ is an $n$-dimensional even polytope, $n\geq 4$.
 If there exists a facet
 $F$ of $P$ with $|V(P)| = 2|V(F)|$, then there exists a $3$-face of $P$
 which is isomorphic to a $3$-dimensional cube.
\end{cor}
\begin{proof}
  It is well known that any $3$-dimensional simple polytope must have a $2$-face $f$ with
  less than $6$ vertices. Now since $P$ is even, any $2$-face of $P$ must have an even number of
  vertices. So there exists a $4$-gon face $f$ in $F$. Then since $|V(P)| = 2|V(F)|$, we have
  $P=F\times [0,1]$ by Corollary~\ref{Cor:bound}.
  So $P$ has a $3$-face $f\times [0,1]$ which is a $3$-cube.
\end{proof}
 \vskip .6cm

 \section{Doubly-even binary codes}

A binary code $C$ is called \emph{doubly-even}
  if the Hamming weight of any codeword
 in $C$ is divisible by $4$. Doubly-even self-dual codes are of special importance among binary
 codes and have been extensively studied. According to Gleason~\cite{gl}, the length of any doubly-even self-dual code is divisible by $8$. In addition,
 Mallows-Sloane~\cite{ms} showed that if $C$ is a double-even self-dual code
 of length $l$, it is necessary that the minimum distance $d$ of $C$ satisfies
   $ d \leq
      4 \big[ \frac{l}{24}\big] +4
      $.
And $C$ is called \emph{extremal} if the equality holds.
 \vskip .1cm

 A result of Zhang~\cite{zhang} tells us that
 an extremal doubly-even self-dual binary code must have length less or equal to $3928$.
However, the existence of extremal doubly-even self-dual binary codes is only known for the following lengths (see~\cite{h} and~\cite[p.273]{rs})
 $$l = 8, 16, 24, 32, 40, 48, 56, 64, 80, 88, 104, 112, 136.$$
 For example, the extended Golay code $\mathcal{G}_{24}$
 is the only doubly-even self-dual [24,12,8] code, and
 the extended quadratic residue code $QR_{48}$ is the only doubly-even self-dual
  [48,24,12] code (see~\cite{hltp}). In addition,
  the existence of an extremal doubly-even self-dual
code of length $72$ is a long-standing open question
 (see~\cite{s} and~\cite[Section 12]{rs}). \vskip .1cm

The following proposition is an immediate consequence of Corollary~\ref{Cor:Main-2}, which gives
us a way to construct doubly-even self-dual codes from simple polytopes.

\begin{prop} \label{Prop:3-face}
     For an $(2k+1)$-dimensional even polytope $P$, the self-dual binary code
     $\mathfrak{B}_{k}(P)$ is doubly-even if and only if
     the number of vertices of any $(k+1)$-dimensional face of $P$
     is divisible by $4$.
\end{prop}

\begin{defn}
 We say that a self-dual binary code $C$ can be \emph{realized by an even polytope} if
 there exists a $(2k+1)$-dimensional
even polytope $P$ so that $C = \mathfrak{B}_{k}(P)$.
\end{defn}
\vskip .1cm

\begin{exam}
 An extremal doubly-even self-dual binary code of length $8$ and $16$ can be realized by
 the $3$-cube and the $8$-prism ($8$-gon$\times [0,1]$), respectively.
 In addition, the $(2k+1)$-dimensional cube realizes a special doubly-even
 \emph{Reed-Muller code}.
\end{exam}

\vskip .1cm

\begin{prop}\label{golay}
The $[24, 12,8]$ extended Golay code $\mathcal{G}_{24}$ cannot be realized by any even polytope.
\end{prop}
\begin{proof}
Assume $\mathcal{G}_{24}$ can be realized by an $n$-dimensional even polytope
 $P^n$, where $n$ is odd. Then $P^n$ has 24 vertices.
By Corollary~\ref{min}, we have $24\geq 2^n$ which implies $n=1,3$.
But $n=1$ is clearly impossible. And by Proposition~\ref{prop:3-polytope},
$n=3$ is impossible either. \vskip .1cm

Another way to prove this result is that since $P^3$ must have a $4$-gon face,
the code $\mathfrak{B}_1(P^3)$ must have a codeword with Hamming
weight $4$. But it is known that
any codeword of $\mathcal{G}_{24}$ must have Hamming weight
$0, 8, 12, 16$ or $24$.
\end{proof}

\vskip .1cm

\begin{prop}\label{quadratic}
The $[48, 24,12]$ extended quadratic residue code
 $QR_{48}$ cannot be realized by any even polytope.
\end{prop}
\begin{proof}
Suppose $QR_{48}$ can be realized by an $n$-dimensional even polytope
 $P^n$. Then by Corollary~\ref{min}, we must have $n=1,3$ or $5$.
 But by Proposition~\ref{prop:3-polytope}, $n$ cannot be $1$ or $3$.
 If $n=5$, since $|V(P^5)|=48$, any $3$-face of $P^5$ has to be an even polytope with $12$
 vertices by Corollary~\ref{Cor:bound} and the fact that the minimum distance of $QR_{48}$ is $12$. Then $P^5$ is isomorphic to the product of a simple
 $3$-polytope with $[0,1]^2$ by Corollary~\ref{Cor:bound} again. This implies that $P^5$ has a
  $3$-face isomorphic to a $3$-cube. But this contradicts the fact that any $3$-face of $P^5$ has
  $12$ vertices.
\end{proof}

 \vskip .1cm

\begin{prop}
 An extremal doubly-even self-dual codes of length $72$ (if exists) cannot be
 realized by any even polytope.
\end{prop}
 \begin{proof}
  Assume that $C$ is an extremal doubly-even self-dual
   binary code of length $72$ and $C$
  can be realized by an even polytope $P$. Then by the definition of extremity,
  the minimum distance of $C$ is $16$ and $P$ has $72$ vertices. Moreover, we have
 \begin{itemize}
  \item[(i)] the dimension of $P$ has to be $5$ by Corollary~\ref{min} and
     Proposition~\ref{prop:3-polytope}; \vskip .1cm

  \item[(ii)] any $3$-face of $P$ must be an even polytope with $16$ vertices by
Corollary~\ref{Cor:bound} and Proposition~\ref{Prop:3-face}.
\end{itemize}
 Then any $4$-face of $P$ must have $32$ or $36$ vertices by Corollary~\ref{Cor:bound}.\vskip .1cm

$\bullet$ If $P$ has a $4$-face $F$ with $32$ vertices, then $F=f\times [0,1]$ where $f$ is a $3$-face
 with $16$ vertices by (ii) and Corollary~\ref{Cor:bound}. This implies that
 $P$ has a $3$-face isomorphic to a $3$-cube by Corollary~\ref{Cor:3-face}. But this
  contradicts (ii).\vskip .1cm

$\bullet$ If $P$ has a $4$-face $F$ with $36$ vertices, then
$P = F\times [0,1]$ by Corollary~\ref{Cor:bound}. So
 $P$ has a $3$-face isomorphic to
  a $3$-cube by Corollary~\ref{Cor:3-face}. This contradicts (ii) again. \vskip .1cm

 So by the above argument, such an even polytope $P$ does not exist.
\end{proof}

\vskip .6cm

\section*{Acknowledgements}
This paper is inspired by a question on
self-dual binary codes and small covers, which was proposed to the authors
by professor Matthias Kreck during his visit to China in September 2012.

\end{document}